\documentclass[dvipdfmx,12pt]{article}
\usepackage[dvips]{graphicx}
\usepackage{amsmath,amssymb,amsthm}
\usepackage{bm}
\usepackage{latexsym}
\usepackage{cases}
\usepackage{tikz}
\usetikzlibrary{intersections,calc,arrows.meta}

\newtheorem{dfn}{Def}[section]
\newtheorem{thm}[dfn]{Theorem}
\newtheorem{prop}[dfn]{Proposition}
\newtheorem{lem}[dfn]{Lemma}
\newtheorem{rem}[dfn]{Remark}
\newtheorem{cor}[dfn]{Corollary}

\makeatletter

\newcounter{myc}[section]

\newcounter{mycc}[section]

\def\id#1{\def\@id{#1}}
\def\department#1{\def\@department{#1}}
\def\superadvisor#1{\def\@superadvisor{#1}}

\makeatletter
    
    \@addtoreset{equation}{section}
\makeatother

\allowdisplaybreaks[4]

\begin{document}

\title{
\vspace{-5ex}
\large{
\bf Local energy decay of solutions to \\ 
the linearized compressible viscoelastic system \\ around  motionless state in an exterior domain}}
\author{
\large{Yusuke Ishigaki ${}^1$ and Takayuki Kobayashi ${}^2$}}
\date{\small
Division of Mathematical Science,\\[0.5ex]
Department of Systems Innovation,\\[0.5ex]
Graduate School of Engineering Science, \\[0.5ex]
The University of Osaka, \\[0.5ex]
Toyonaka-shi Osaka 560-8531, Japan \\[-3.5ex]
\begin{align*}
&{}^1 \textrm{e-mail:} \texttt{ y.ishigaki.es@osaka-u.ac.jp} \\
&{}^2 \textrm{e-mail:} \texttt{ kobayashi@sigmath.es.osaka-u.ac.jp}
\end{align*}
}

\maketitle

\begin{abstract}
We study the large time behavior of solutions to the system of equations describing motion of compressible viscoelastic ﬂuids. We focus on the linearized system around a motionless state in a three-dimensional exterior domain and derive the local energy decay estimate of its solution to give the diﬀusion wave phenomena caused by sound wave, viscous diffusion and elastic shear wave.
\end{abstract}

\vspace{1ex}

\section{Introduction}
\subsection{Problems}
This paper deals with the compressible viscoelastic system   
\begin{gather}
\partial_t\rho+\mathrm{div} (\rho v)=0, \label{system-0-1} \\
\rho(\partial_t v+ v\cdot\nabla v)
-\nu{\Delta}v -(\nu+\nu')\nabla \mathrm{div} v
+\nabla  P(\rho)  = \beta^2\mathrm{div}(\rho F{}^\top\! F), \label{system-0-2} \\
\partial_t F + v\cdot\nabla F =(\nabla v)F \label{system-0-3}
\end{gather}
in an exterior domain $\Omega$ in $\mathbb{R}^3$ with smooth compact boundary $\partial \Omega$.
Here $\rho=\rho(x,t)$, 
$v ={}^\top (v^1(x,t),$ $ v^2(x,t), v^3(x,t))$, and $F=(F^{jk}(x,t))_{1\leq j,k\leq 3}$ denote
 the unknown density, the velocity field, and the deformation tensor, respectively, at  position $x\in \Omega$ and time $t\geq 0$;
$P=P(\rho) $ is the pressure assumed to be a smooth function of $\rho>0$ satisfying $P^\prime(1)>0$; $\nu$ and $\nu'$ are the viscosity coefficients satisfying $\nu>0$ and $2\nu+3\nu'\geq0;$
$\beta>0$ is a constant called propergation of elastic shear wave. In particular, if we set $\beta=0$ formally, \eqref{system-0-1} and \eqref{system-0-2} become the usual compressible Navier-Stokes equations.  

We next impose the boundary and initial conditions 
\begin{gather} 
v|_{\partial \Omega} =0,~(\rho, v, F)\to(1,0,I)~\textrm{as}~|x|\to \infty, \label{BC-0} \\
(\rho, v, F)|_{t=0}=(\rho_0,v_0,F_0), \label{IC1-0}
\end{gather}
where $I$ is the $3\times3$ identity matrix.
Furthermore, we also assume the following conditions for $\rho_0$ and $F_0$:
\begin{align}
\mathrm{div}(\rho_0 {}^\top\! F_0)=0, \quad
\rho_0\mathrm{det}F_0=1, \quad
\sum_{m=1}^3(F_0^{ml}\partial_{x_m}F_0^{jk}-F_0^{mk}\partial_{x_m}F_0^{jl})=0,
~(j,k,l=1,2,3). 
\label{IC2-0}
\end{align}
As in the proof of \cite[Proposition.1]{QZ10}, \eqref{IC2-0} is invariant for $t$ and $x$.
These constraints are necessary to justify the global existence of solutions to \eqref{system-0-1}--\eqref{IC1-0} (See e. g. \cite{CW18, HW11, Q11}.) The first constraint in \eqref{IC2-0} is used to perturb \eqref{system-0-1}--\eqref{system-0-3} around given flow, second one preserves a local volume in material without the vacuum $\rho>0$ of the fluid, and third one stands for Piola's formula which follows from the inverse matrix of $F$ written by a Jacobi matrix of some vector field.

Throughout this paper, we are interested in the behavior of solutions to the problem \eqref{system-0-1}--\eqref{IC1-0} with \eqref{IC2-0} around a motionless state $(\rho,v,F)=(1,0,I)$, where $I$ denotes a 3 $\times$ 3 identity matrix and stands for no deformation of elastic body. To investigate this, we concentrate a behavior of the perturbation 
$u(t):={}^\top(\phi(t),w(t),G(t))={}^\top(\rho(t),v(t),F(t))-{}^\top(1,0,I)$ satisfying the following initial boundary value problem
\begin{gather}
\partial_t u+A u=g(u), \quad
u|_{t=0}=u_0=(\phi_0,w_0,G_0),  \quad
w|_{\partial \Omega}=0, \quad u\to0~\mathrm{as}~|x|\to\infty. \label{IBV-1}
\end{gather}
Here $A$ is a linearized operator around $(1,0,I)$ given by
\begin{equation} \label{def-L}
A:=
\left(
\begin{array}{@{\ }cc@{\ }cc@{\ } }
0 & \mathrm{div} & 0\\
\gamma^2\nabla & -\nu\Delta-\tilde{\nu}\nabla\mathrm{div} & -\beta^2\mathrm{div} \\
0 & -\nabla & 0
\end{array}
\right),
\end{equation}
where $\gamma:=\sqrt{P^\prime(1)}$ and $\tilde{\nu}:=\nu+\nu'$; $g(u)={}^\top (g_1(u),g_2(u),g_3(u))$ is a nonlinear term given by
\begin{align*}
g_1(u)&:=-\mathrm{div}(\phi w), \\
g_2(u)&:=-w\cdot\nabla w+\frac{\phi}{1+\phi}(-\nu\Delta w-\tilde{\nu}\nabla\mathrm{div}w+\gamma^2\nabla \phi)-\frac{1}{1+\phi}\nabla Q(\phi)\\
&\quad\quad-\frac{\beta^2\phi}{1+\phi}\mathrm{div} G +\frac{\beta^2}{1+\phi}\mathrm{div}(\phi G+G{}^\top\! G+\phi G{}^\top\! G), \\
g_3(u)&:=-w\cdot\nabla G+(\nabla w) G, 
\end{align*}
where $Q(\phi):=\phi^2\int_0^1 {P}^{\prime\prime}(1+s\phi)\mathrm{d}s $. The aim of this paper is to clarify the large time behavior of the solution to the corresponding linearized problem to \eqref{IBV-1}:
\begin{gather}
\partial_t u +Au = 0, \quad
u|_{t=0}=u_0=(\phi_0,w_0,G_0), \quad
w|_{\partial \Omega}=0, \quad u\to0~\mathrm{as}~|x|\to\infty \label{L-IBV-1} 
\end{gather}
under the constraints obtained by linearizing \eqref{IC2-0} without its second constraint:
\begin{gather}
\nabla \phi_0+ \mathrm{div}{}^\top\! G_0 =0,  \quad
\partial_{x_l} G_0^{jk} =\partial_{x_k} G_0^{jl}~(j,k,l=1,2,3). \label{L-Constr-0}
\end{gather}
For simplicity, we express the solution of \eqref{L-IBV-1} with initial data $u_0$ as $u(t)=:T(t)u_0$ and consider the case that there exists a vector field $\psi$ such that $G$ is written as $G=\nabla \psi$, which automatically satisfies the last identity of \eqref{L-Constr-0}.

The system \eqref{system-0-1}--\eqref{system-0-3} is governed from a motion of viscous compressible fluid with effect of linear Hookean elastic body. The equations \eqref{system-0-1} and \eqref{system-0-3} are obtained under the definition of $\rho$ and $F$ in microscopic scale, and \eqref{system-0-2} is a motion equation derived by the variational method. We refer to \cite{G81,LLZ05, SiTh04} for more physical details. Moreover, we can categorize  \eqref{system-0-1}--\eqref{system-0-3} as a quasilinear parabolic-hyperbolic system since \eqref{system-0-1} and \eqref{system-0-3} are regarded as  quasilinear transport equations and \eqref{system-0-2} is seen as a parabolic equation for $v$.   

\subsection{Known results}

We first review the known results for mathematical study of solutions $u=(\phi,w)$ of \eqref{IBV-1} with $\beta=0$ around the motionless state $(1,0)$. Matsumura and Nishida \cite{MN79} showed the global existence and $H^2(\mathbb{R}^3)$ decay estimate of $u$. Hoff and Zumbrun \cite{HZ95} obtained the $L^p(\mathbb{R}^n)~(1\le p\le \infty,~n\ge2)$ estimates of $u$, which imply that the behavior of $u$ in $L^p(\mathbb{R}^n)$ is similarly described as the heat kernel for $p>2$ and becomes different from that for $p<2$.
See also Kobayashi and Shibata \cite{KS02} for the linearized problem whose solution behaves like a sum of purely diffusive part and combination of fundamental solutions to the diffusion equation and wave equation with propagation speed $\gamma$ as $t\to\infty$.  Kobayashi \cite{K97}, Kobayashi and Shibata \cite{KS99}, and Enomoto and Shibata \cite{ES12,SE18} studied the asymptotic behavior of the linearized semigroup in an exterior domain, and Kobayashi \cite{K02-1,K02-2} and Kobayashi and Shibata \cite{KS99} also applied these facts to establish the $L^p~(1\le p\le \infty)$ estimates of solutions to the nonlinear problem in a three dimensional exterior domain.

We next give the known results for mathematical study of solutions $u=(\phi,w,G)$ of \eqref{IBV-1}. This is widely investigated in the whole space case $\mathbb{R}^3$ as follows. Sideris and Thomases \cite{SiTh04} firstly employed \eqref{IBV-1} to prove the global-in-time existence of solutions to incompressible viscoelastic system in $\mathbb{R}^3$ via incompressible limit.  Hu and Wu \cite{HW13} guaranteed the global existence of solution in $H^2(\mathbb{R}^3)$ framework under small initial perturbation in $H^2(\mathbb{R}^3)$, and derived its $L^p(\mathbb{R}^3)$ decay estimates for $2\leq p\leq 6$. Li, Wei and Yao \cite{LWY16} extended their result to show the global existence of $u$ in $H^N(\mathbb{R}^3)~(N\ge3)$ framework under small initial perturbation in $H^N(\mathbb{R}^3)$ and its $L^p(\mathbb{R}^3)$ decay estimates for $6< p\leq \infty$. Their obtained decay rates is same as the heat kernel and reflects the diffusion phenomena of \eqref{IBV-1} observed from viscous effect.  However, these obtained rates do not expose the wave phenomena originating from sound and elastic wave. To discover both phenomena by clarifying the large time behavior of $u$, Ishigaki \cite{I20,I22} refined and extended their obtained $L^p(\mathbb{R}^3)$ decay estimates for $1\le p\leq \infty$.  This is observed from the structure of the linearized semigroup.  Its incompressible part behaves like diffusion wave which consists of fundamental solutions to the diffusion equation and wave equation with propagation speed $\beta$, while its compressible part is effected by diffusion wave constructed by fundamental solutions to the diffusion equation and wave equation with propagation speed $\sqrt{\beta^2+\gamma^2}$. For more mathematical studies, we refer  \cite{BZ23,WLY16,WGT17} in the Sobolev framework and \cite{H18,LWXZ16,PX19,QZ10} in the Besov framework. For the mathematical study of \eqref{IBV-1} in an general domain $\Omega$, the global existence of the solution in $H^2(\Omega)$ framework is established by Qian \cite{Q11} for the case that $\Omega$ is half space, exterior domain and bounded domain, provided that  $u_0$ is small in $H^2(\Omega)$. Moreover, Qian \cite{Q11} also proved that if $\Omega$ is bounded and $\int_\Omega \phi_0dx=0$ holds, then the solution decays exponentially as $t\to\infty$. See also the detailed proof of Chen and Wu \cite{CW18} for the bounded domain case. Hu and Wang \cite{HW15} showed this global existence result in $L^p(\Omega)$ framework with a bounded domain $\Omega$. However, there is no results for the large time behavior of u in an unbounded domain $\Omega$ except $\mathbb{R}^3$.

\subsection{Main results}

In this subsection, we give the main result of this paper. For $R>0$ and $D\subset \mathbb{R}^3$, we define $B_R$ by $B_R:=\{x\in\mathbb{R}^3~|~|x|< R\}$ and then set $D_R:=D\cap B_R$. We obtain the local energy decay estimates of solution to \eqref{L-IBV-1} in this paper. This detail is stated as the following theorem.
\begin{thm}\label{Local-energy}
Let $1<p<\infty$, let $b_0>0$ be $\mathbb{R}^3\setminus\Omega\subset B_{b_0}$ and let $b>b_0$. If $u_0=(\phi_0,w_0,G_0)$ satisfies $u_0\in W^{1,p}(\Omega)\times L^p(\Omega)^3\times W^{1,p}(\Omega)^{3\times3}$, $\nabla\phi_0+\mathrm{div}{}^\top G_0=0$, $G_0\in \nabla(W^{2,p}(\Omega)\cap W_0^{1,p}(\Omega))^3$ and $u_0(x)=0~(x\notin\Omega_b)$, then the following local energy decay estimates of $u(t)=T(t)u_0$ hold for all $m=0,1,2,\ldots$ and $t\ge 1$:
\begin{gather}\label{Local-energy-1}
\begin{aligned}
&\|\partial_t^m T(t)u_0\|_{W^{1,p}(\Omega_b)\times W^{2,p}(\Omega_b)^3\times W^{1,p}(\Omega_b)^{3\times3}} 
&\leq C_{b,b_0,m,p} ~ t^{-2-m}\|u_0\|_{W^{1,p}(\Omega)\times L^p(\Omega)^3\times W^{1,p}(\Omega)^{3\times3}}, 
\end{aligned}
\end{gather}
where $C_{b,b_0,m,p}$ is a constant depending only on $b,b_0,m$ and $p$.
\end{thm}
We mention the notable remark of this obtained result. For the case $\beta=0$, Kobayashi \cite{K97} established the local energy decay of the linearized semigroup in $\Omega$. Its  decay rate is given by $t^{-3/2-m}$ which coincides the heat kernel. On the other hand, for the case $\beta=0$, the obtained decay rate in the right-hand side of \eqref{Local-energy-1} is $t^{-1/2}$ faster than the case $\beta=0$. This provides the different aspect from the case $\beta=0$ due to the diffusion wave induced by interaction between sound wave, viscous diffusion and elastic wave, analogous to the results in \cite{I20,I22} for the whole space case.

\subsection{Outline of Theorem \ref{Local-energy}}
We explain the proof of Theorem \ref{Local-energy} briefly in this subsection. The proof of Theorem \ref{Local-energy} is based on the corresponding resolvent problem to \eqref{L-IBV-1}:
\begin{equation}
\label{R-Ex-1}
\lambda u+Au=\mathbf{f}~\text{in}~\Omega, ~
w|_{\partial \Omega}=0.
\end{equation}
Here $\lambda\in\mathbb{C}$ is a resolvent parameter;  $\mathbf{f}={}^\top(f_1,f_2,f_3)$ is a given pair of three functions $f_1=f_1(x)\in\mathbb{R},~f_2=f_2(x)\in\mathbb{R}^3$ and $f_3=f_3(x)\in\mathbb{R}^{3\times3}$ satisfying $\nabla f_1+\mathrm{div} {}^\top f_3=0$ and $f_3\in \nabla(W_0^{1,p}(\Omega))^3$.
We note that $f_3\in \nabla(W_0^{1,p}(\Omega))^3$ immediately leads to the second condition of \eqref{L-Constr-0} with $G=f_3$, and plays a key role in \cite{CW18,I20, I22, Q11}, which considered the situation $G \sim \nabla \psi$ under the small initial perturbation. Here $\psi$ is some vector field belonging to $(W^{2,2}(\Omega)^{3} \cap W_0^{1,2}(\Omega))^{3}$.
We then see that the resolvent set contains some sectorial domain. Furthermore, it is well-known that $A$ generates an analytic semigroup $e^{-tA}$ on $W^{1,p}(\Omega)\times L^{p}(\Omega)^{3\times3} \times W^{1,p}(\Omega)^{3\times3}$ for $1<p<\infty$. To investigate asymptotic behavior of $e^{-tA}$ as $t\to\infty$, the information of the solution to \eqref{R-Ex-1} near $\lambda=0$ plays a crucial role and therefore we carry out the following sections, inspired by \cite{K97,SE18}. In Section 2, we introduce several notations and crucial lemmata. In Section 3, we consider the corresponding resolvent problem in $\mathbb{R}^3$ and capture the information of its solution. It is shown that the singular part of its solution as $\lambda\to0$ is described as $\lambda^2\log \lambda$. On the other hand, when $\beta=0$, this is only majored as $\lambda^{1/2}$ in \cite{K97}. These facts give the difference between two cases $\beta=0$ and $\beta>0$. In Section 4, we study the corresponding zero resolvent problem in an bounded domain.  In Section 5, we focus on \eqref{R-Ex-1} and approximate its solution via parametrix method by using a cut-off function and the results in Sections 3 and 4. The difficulty of the proof is keeping the constraints of $f_1$ and $f_3$. To overcome this, we rely on the non-local operators to modify the parametrix, and then obtain the singularity of the solution to \eqref{R-Ex-1} as $\lambda\to0$. In Section 6, we derive Theorem \ref{Local-energy} by complex integral and oscillating integral theory.

\section{Notation}
In this section, we prepare notations and function spaces which will be used throughout the paper. Let $\mathbb{N}$ be a set of any positive integers and let $\mathbb{N}_0:=\{0\}\cup \mathbb{N}$. Throughout this section, we use $D$ as a domain assumed to be $\mathbb{R}^3$, bounded domain or exterior domain in $\mathbb{R}^3$, and we regard $p$ and $m$ as $1\leq p\leq \infty$ and $m\in\mathbb{N}_0$, respectively.

The symbol $L^p(D)$ denotes the usual Lebesgue space on $D$, and its norm is denoted by $\|\cdot\|_{L^p(D)}$. 
Similarly $W^{m,p}(D)$ denotes the $m$-th order $L^p$ Sobolev space on $D$,~and its norm is denoted by $\|\cdot\|_{W^{m,p}(D)}$.  For the case $p=2$, we set $H^m(D):=W^{m,2}(D)$. 
We define $W_0^{1,p}(D)$ as the complition of $C_0^\infty(D)$ in $W^{1,p}(D)$, where $C_0^\infty(D)$ denotes the set of all $C^\infty (D)$ functions whose supports are compact in $D$. We call $L_{\mathrm{loc}}^p(D)$ and $W_{\mathrm{loc}}^{m,p}(D)$ the locally Lebesgue space on $D$ and locally $m$-th order $L^p$ Sobolev space on $D$, respectively.

Partial derivatives of a function $u$ in $x_j~(j=1,2,3)$ and $t$ are denoted by $\partial_{x_j}u$ and $\partial_t u$, respectively. The symbol $\Delta$ denotes the usual Laplacian with respect to $x$. For  a multi-index $\alpha=(\alpha_1,\alpha_2,\alpha_3)\in(\mathbb{N}_0)^3$ and $\xi={}^\top(\xi_1,\xi_2,\xi_3)\in\mathbb{R}^3$, we define $|\alpha|$, $\partial_x^\alpha$ and $\xi^\alpha$ by $|\alpha|:=\alpha_1+\alpha_2+\alpha_3$, $\partial_x^\alpha :=\partial_{x_1}^{\alpha_1} \partial_{x_2}^{\alpha_2} \partial_{x_3}^{\alpha_3}$ and $\xi^\alpha:=\xi_1^{\alpha_1}\xi_2^{\alpha_2}\xi_3^{\alpha_3}$, respectively. For a function $u$ and a non-negative integer $k$, $\nabla^k u$ stands for $\nabla^k u:=\{\partial_x^\alpha u~|~|\alpha|=k\}$ and its $W^{m,p}(D)$ norm is denoted by $\|\nabla^k u\|_{W^{m,p}(D)}:=\sum_{|\alpha|=k}\|\partial_x^\alpha u\|_{W^{m,p}(D)}$.

For a scalar valued function $\rho=\rho(x)$, we denote by $\nabla\rho$ its gradient with respect to $x$. For a vector valued function $v=v(x)={}^\top(v^1(x),v^2(x),v^3(x))$, we denote by $\mathrm{div}v$ and $(\nabla v)^{j,k}=(\partial_{x_k}v^j)$ its divergence and Jacobian matrix with respect to $x$, respectively, and we set $\|v\|_{W^{m,p}(D)}:=\sum_{j=1}^3\|v^j\|_{W^{m,p}(D)}$. Here the symbol ${}^\top\cdot$ stands for the transposition. For a $3\times3$-matrix valued function $F=F(x)=(F^{j,k}(x))_{1\le j,k\le 3}$, we define its divergence $\mathrm{div}F={}^\top((\mathrm{div}F)^1,(\mathrm{div}F)^2,(\mathrm{div}F)^3)$ and trace $\mathrm{tr}F$ by $(\mathrm{div}F)^j=\sum_{k=1}^3 \partial_{x_k}F^{j,k}$ and $\mathrm{tr}F=\sum_{k=1}^3 F^{kk}$, respectively, and we set $\|F\|_{W^{m,p}(D)}:=\sum_{j,k=1}^3\|F^{j,k}\|_{W^{m,p}(D)}$.

For Banach spaces $X,~Y$ and domain $J\subset \mathbb{C}$, we define $\mathcal{L}(X,Y)$ as a set of any bounded linear operators $T:X\to Y$, and $\mathcal{A}(J,X)$ stands for a set of any $X$-valued bounded holomorphic functions in $J$. For simplicity, we write $\mathcal{L}(X):=\mathcal{L}(X,X)$.

We denote the Fourier transform of a scalar-valued function $f=f(x)$ by $\widehat{f}$ or $\mathcal{F}f$:
\[
\widehat{f}(\xi)=(\mathcal{F}f)(\xi):=\frac{1}{(2\pi)^{\frac{3}{2}}}\int_{\mathbb{R}^3} f(x)e^{-i\xi\cdot x} \mathrm{d}x~(\xi\in\mathbb{R}^3).
\] 
The Fourier inverse transform is denoted by $\mathcal{F}^{-1}$:
\[
(\mathcal{F}^{-1}f)(x):=\frac{1}{(2\pi)^{\frac{3}{2}}}\int_{\mathbb{R}^3} f(\xi)e^{i\xi\cdot x} \mathrm{d}\xi~(x\in\mathbb{R}^3).
\] 
Similarly, for vector valued function $v={}^\top(v^1,v^2,v^3)$ and matrix valued function $F=(F^{jk})_{1\le j,k\le 3}$, their Fourier transforms and Fourier inverse transforms are given by 
\begin{align*}
\widehat{v}(\xi)&:=(\mathcal{F}v)(\xi):={}^\top((\mathcal{F}v^1)(\xi),(\mathcal{F}v^2)(\xi),(\mathcal{F}v^3)(\xi)),\\
(\mathcal{F}^{-1}v)(x)&:={}^\top((\mathcal{F}^{-1}v^1)(x),(\mathcal{F}^{-1}v^2)(x),(\mathcal{F}^{-1}v^3)(x)), \\
\widehat{F}(\xi)&:=(\mathcal{F}F)(\xi):=((\mathcal{F}F^{j,k})(\xi))_{1\le j,k\le 3}, \quad
(\mathcal{F}^{-1}F)(x):=((\mathcal{F}^{-1}F^{j,k})(x))_{1\le j,k\le 3}.
\end{align*}
We also define the non-local operator $(-\Delta_{\mathbb{R}^3})^{-1}$ as $(-\Delta_{\mathbb{R}^3})^{-1}:=\mathcal{F}^{-1}|\xi|^{-2}\mathcal{F}$.

We next introduce several symbols for setting the solution spaces of the corresponding resolvent problem. Let $p\in[1,\infty]$, $R>0$ and $m\in\mathbb{N}_0$. We then introduce the following function spaces
\begin{align*}
&
\begin{aligned}
X^{p}(D):=&\left\{(\phi,w,G)\in W^{1,p}(D)\times L^{p}(D)^3\times W^{1,p}(D)^{3\times3}~|~ \right. \\
&\left.~\nabla \phi+\mathrm{div}{}^\top G=0,~G\in \nabla (W^{2,p}(D)\cap W_0^{1,p}(D))^3\right\},
\end{aligned}
\\
&
\begin{aligned}
Y^{p}(\mathbb{R}^3):=&\left\{(\phi,w,G)\in W^{1,p}(\mathbb{R}^3)\times W^{2,p}(\mathbb{R}^3)^3\times W^{1,p}(\mathbb{R}^3)^{3\times3}~|~ \right. \\
&\left.~\nabla \phi+\mathrm{div}{}^\top G=0,~G\in \nabla W^{2,p}(\mathbb{R}^3)^3\right\},
\end{aligned}
\\
&
\begin{aligned}
Y^{p}(D):=&\left\{(\phi,w,G)\in W^{1,p}(D)\times W^{2,p}(D)^3\times W^{1,p}(D)^{3\times3}~|~ \right. \\
&\left.~\nabla \phi+\mathrm{div}{}^\top G=0,~w|_{\partial D}=0,~G\in \nabla (W^{2,p}(D)\cap W_0^{1,p}(D))^3\right\}~(D\neq\mathbb{R}^3),
\end{aligned}
\\
&
L_R^p(D):=\{f\in L^p (D)|~f(x)=0~(x\notin D_R)\}, \quad
W_R^{k,p}(D):=\{f\in W^{k,p} (D)|~f(x)=0~(x\notin D_R)\}, \\
&
X_R^p(D):=\{f\in X^p (D)|~f(x)=0~(x\notin D_R)\},
\end{align*}
and the norms of $u=(\phi,w,G)$ in $X^p(D)$ and $Y^p(D)$:
\begin{align*}
\|u\|_{X^p(D)}&:=\|\phi\|_{W^{1,p}(D)}+\|w\|_{L^p(D)}+\|G\|_{W^{1,p}(D)}, \quad
\|u\|_{Y^p(D)}:=\|\phi\|_{W^{1,p}(D)}+\|w\|_{W^{2,p}(D)}+\|G\|_{W^{1,p}(D)}.
\end{align*}

We next give several sets in $\mathbb{C}$ for stating the results of the resolvent problem. Let $\varepsilon\in (0,\pi/2)$ and $\delta,\lambda_0>0$, and let $\Xi$ and $\Xi_\varepsilon$ be
\begin{align*}
\Xi&:=\left\{\lambda\in\mathbb{C}~|~(\mathrm{Re}\lambda+R_0)^2+(\mathrm{Im}\lambda)^2>R_0^2\right\},~R_0:=\max\left\{\frac{\beta^2}{\nu},\frac{\beta^2+\gamma^2}{\nu+\tilde{\nu}}\right\}, \\
\Xi_\varepsilon&:=\left\{\lambda\in\mathbb{C}~|~(\mathrm{Re}\lambda+R_0+\varepsilon)^2+(\mathrm{Im}\lambda)^2\ge(R_0+\varepsilon)^2\right\}.
\end{align*}
We then set
\begin{align*}
\Sigma_{\varepsilon}&:=\{\lambda\in\mathbb{C}\setminus\{0\}~|~|\arg \lambda|\le \pi-\varepsilon\},~ \Sigma_{\varepsilon, \delta} :=\{\lambda\in\mathbb{C}~|~\lambda-\delta\in \Sigma_\varepsilon\}, \\
V_{\varepsilon,\lambda_0}&:=\{\lambda\in\Sigma_{\varepsilon}\cap \Xi~|~|\lambda|>\lambda_0\}, ~
U_{\lambda_0}:=\{\lambda\in \mathbb{C}~|~|\lambda|<\lambda_0\},
~
\Dot{U}_{\lambda_0}:=\{\lambda\in \mathbb{C}~|~\mathrm{Re}\lambda\ge0,~0<|\lambda|<\lambda_0\}.
\end{align*}
Here $\arg \lambda\in(-\pi,\pi]$ denotes an argument of $\lambda$.

\vskip2mm

We finally introduce two important lemmata to investigate the corresponding resolvent problem. The first lemma is the Fourier multiplier theorem which plays a crucial role for studying the resolvent problem in $\mathbb{R}^3$. 
\begin{lem}\label{FMthm}
Let $1<p<\infty$, let $\mathcal{S}(\mathbb{R}^3)$ be the Schwartz space defined as a set of any rapidly decreasing $C^\infty(\mathbb{R}^3)$ functions, and let $m=m(\xi)$ be a smooth function in $\mathbb{R}^3 \setminus \{0\}$ satisfying the multiplier condition
\begin{align*}
|\partial_\xi^\alpha m(\xi)|\le C_\alpha |\xi|^{-\alpha}
\end{align*}
for any multiindex $\alpha\in\mathbb{N}_0^3~(|\alpha|\le 2)$ and $\xi\in\mathbb{R}^3\setminus\{0\}$. Then the operater $T:\mathcal{S}(\mathbb{R}^3)\to \mathcal{S}(\mathbb{R}^3)$ given by $(Tf)(x):=\mathcal{F}^{-1}[m(\xi)\widehat{f}(\xi)](x)$ is able to be extended as $\mathcal{L}(L^p(\mathbb{R}^3))$. Furthermore, $T$ satisfies the estimate
\[
\|Tf\|_{L^p(\mathbb{R}^3)}\le C_p\max_{\alpha\in\mathbb{N}_0^3~(|\alpha|\le 2)} C_\alpha\|f\|_{L^p(\mathbb{R}^3)}
\]
for any $f\in L^p(\mathbb{R}^3)$.
\end{lem}
We next prepare the second lemma which assures the solvability of the elastic equation with complex coefficient to study the  resolvent problem in a bounded domain. See the example for \cite[Proposition 2.1]{SWZ11}.

\begin{lem} \label{bdd-elliptic}
Let $a>0$, $b\ge0$, $m\in \mathbb{N}_0$, $1<p<\infty$, and let $D$ be a bounded domain with smooth boundary. Then for any $f\in W^{m,p}(D)^3$, there exists a unique solution $w\in W^{m+2,p}(D)^3\cap W_0^{1,p}(D)^3$ of the Dirichlet problem
\begin{equation}
-a\Delta w-b\nabla \mathrm{div}w=f~\textrm{in}~D,\quad w|_{\partial D}=0. \label{bdd-elliptic-1}
\end{equation}
Furthermore, $w$ satisfies 
\[
\|w\|_{W^{m+2,p}(D)}\leq C\|f\|_{W^{m,p}(D)}
\]
with some constant $C=C(a,b,m,p)$
\end{lem}
For $f\in W^{m,p}(D)^3$, we define $(-\Delta_D)^{-1}f$ as the unique solution of \eqref{bdd-elliptic-1} with $a=1$ and $b=0$.

\section{Resolvent problem in the whole space} 
In this section, we investigate the corresponding resolvent problem to \eqref{L-IBV-1} in the whole space $\mathbb{R}^3$. Let $\lambda\in\mathbb{C}$ be a resolvent parameter and let $\mathbf{f}={}^\top(f_1,f_2,f_3)$ be a given pair of three functions $f_1=f_1(x)\in\mathbb{R},~f_2=f_2(x)\in\mathbb{R}^3$ and $f_3=f_3(x)\in\mathbb{R}^{3\times3}$. Then the resolvent problem in $\mathbb{R}^3$ is given by
\begin{equation}
\label{R-R3-1}
\lambda u+Au=\mathbf{f}~\mathrm{in}~\mathbb{R}^3, \quad
u\to 0~\textrm{as}~|x|\to\infty.
\end{equation}
Here $A$ is given by \eqref{def-L}. We are now going to derive the solution formula of \eqref{R-R3-1}. Applying the Fourier transform to the first equation of \eqref{R-R3-1} with respect to $x$, we have
\begin{equation}
\label{R-R3-1-1}
\begin{aligned}
\lambda \widehat{\phi}+i\xi\cdot\widehat{w}&=\widehat{f_1}, \\
\lambda \widehat{w}+(\nu|\xi|^2+\tilde{\nu}\xi^\top\xi)\widehat{w}+i\gamma^2\widehat{\phi}\xi-i\beta^2\widehat{G}\xi&=\widehat{f_2}, \\
\lambda \widehat{G}-i\widehat{w}^\top\xi&=\widehat{f_3}. 
\end{aligned}
\end{equation} 
We next decompose $\widehat{w}$ as its incompressible part $\widehat{w}_s~(i\xi\cdot\widehat{w}_s=0)$ and compressible part $\widehat{w}_c$:
$
\widehat{w}=\widehat{w}_s+\widehat{w}_c,\quad \widehat{w}_s(\xi):=\widehat{\mathcal{P}}_s(\xi)\widehat{w}(\xi),~\widehat{w}_c(\xi):=\widehat{\mathcal{P}}_c(\xi)\widehat{w}(\xi),
$
where $\widehat{\mathcal{P}}_s(\xi)$ and $\widehat{\mathcal{P}}_c(\xi)$ are given by $\widehat{\mathcal{P}}_s(\xi):=I-\xi^\top\xi/|\xi|^2$ and $\widehat{\mathcal{P}}_c(\xi):=\xi^\top\xi/|\xi|^2$, respectively. We then assume $\lambda\neq0$ once to rewrite the first and third equations of \eqref{R-R3-1-1} as 
\begin{gather}
\widehat{\phi}=\lambda^{-1} (\widehat{f_1}-i\xi\cdot\widehat{w}),\quad \widehat{G}=\lambda^{-1}(\widehat{f_3}+i\widehat{w}^\top\xi).
\label{R-R3-1-2}
\end{gather}
We next substitute \eqref{R-R3-1-2} into the second equation of \eqref{R-R3-1-1} and then apply $\lambda\widehat{\mathcal{P}}_l~(l=s,c)$ to the resultant equation to arrive at
$(\lambda^2+\kappa_1(\lambda)|\xi|^2)\widehat{w}_s=\widehat{\mathcal{P}}_s(\lambda\widehat{f}_2+i\beta^2\widehat{f}_3\xi) 
$
and 
$
(\lambda^2+\kappa_2(\lambda)|\xi|^2)\widehat{w}_c=\widehat{\mathcal{P}}_c(-i\gamma^2\widehat{f}_1\xi+\lambda\widehat{f}_2+i\beta^2\widehat{f}_3\xi),
$
where $\kappa_1(\lambda):=\nu\lambda+\beta^2$ and $\kappa_2(\lambda):=(\nu+\tilde{\nu})\lambda+\beta^2+\gamma^2$. Therefore, assuming $\lambda^2+\kappa_1(\lambda)|\xi|^2\neq0$ and $\lambda^2+\kappa_2(\lambda)|\xi|^2\neq0$ for any $\xi\in\mathbb{R}^3$ additionally, we are able to construct $\widehat{w}$ explicitly. As a result, combining the expression of $\widehat{w}$ and \eqref{R-R3-1-2} implies that $\widehat{u}=(\widehat{\phi},\widehat{w},\widehat{G})$ is derived explicitly.
In fact, we have the solution formula
\begin{align}
&\label{R-R3-phi}
(\Phi_{\mathbb{R}^3}(\lambda)\mathbf{f})(x):=\mathcal{F}^{-1}\left[\frac{\lambda+(\nu+\tilde{\nu})|\xi|^2}{\lambda^2+\kappa_2(\lambda)|\xi|^2}\widehat{f}_1(\xi)\right](x) -i\mathcal{F}^{-1}\left[\frac{\xi\cdot\widehat{f}_2(\xi)}{\lambda^2+\kappa_2(\lambda)|\xi|^2}\right](x)+(\overline{R}_1(\lambda)f)(x),
\\
&\label{R-R3-w}
\begin{aligned}
(\mathcal{W}_{\mathbb{R}^3}(\lambda)\mathbf{f})(x)
&:=-i\gamma^2\mathcal{F}^{-1}\left[\frac{\widehat{f}_1(\xi)}{\lambda^2+\kappa_2(\lambda)|\xi|^2}\xi\right](x) \\
&\quad+\mathcal{F}^{-1}\left[\left(\frac{\lambda}{\lambda^2+\kappa_1(\lambda)|\xi|^2}\widehat{\mathcal{P}}_s(\xi)+\frac{\lambda}{\lambda^2+\kappa_2(\lambda)|\xi|^2}\widehat{\mathcal{P}}_c(\xi)\right)\widehat{f}_2(\xi)\right](x)\\
&\qquad +i\beta^2\mathcal{F}^{-1}\left[\left(\frac{1}{\lambda^2+\kappa_1(\lambda)|\xi|^2}\widehat{\mathcal{P}}_s(\xi)+\frac{1}{\lambda^2+\kappa_2(\lambda)|\xi|^2}\widehat{\mathcal{P}}_c(\xi)\right)\widehat{f}_3(\xi)\xi\right](x),
\end{aligned}
\\
&\label{R-R3-G}
\begin{aligned}
(\mathcal{G}_{\mathbb{R}^3}(\lambda)\mathbf{f})(x) 
&:=i\mathcal{F}^{-1}\left[\left(\frac{1}{\lambda^2+\kappa_1(\lambda)|\xi|^2}\widehat{\mathcal{P}}_s(\xi)+\frac{1}{\lambda^2+\kappa_2(\lambda)|\xi|^2}\widehat{\mathcal{P}}_c(\xi)\right)\widehat{f}_2(\xi)^\top\xi\right](x)\\
&\quad +\mathcal{F}^{-1}\left[\left(\frac{\lambda+\nu|\xi|^2}{\lambda^2+\kappa_1(\lambda)|\xi|^2}\widehat{\mathcal{P}}_s(\xi)+\frac{\lambda+(\nu+\tilde{\nu})|\xi|^2}{\lambda^2+\kappa_2(\lambda)|\xi|^2}\widehat{\mathcal{P}}_c(\xi)\right)\widehat{f}_3(\xi)\right](x)\\
&\qquad+(\overline{R}_3(\lambda)f)(x),
\end{aligned}
\end{align}
where $\overline{R}_1(\lambda)\mathbf{f}$ and $\overline{R}_3(\lambda)\mathbf{f}$ denote
\begin{align*}
&(\overline{R}_1(\lambda)\mathbf{f})(x):=-\frac{\beta^2}{\lambda}\mathcal{F}^{-1}\left[\frac{\widehat{f_1}(\xi)|\xi|^2+{}^\top\xi\widehat{f}_3(\xi)\xi}{\lambda^2+\kappa_2(\lambda)|\xi|^2}\right](x),\\
&
\begin{aligned}
(\overline{R}_3(\lambda)\mathbf{f})(x) 
&:=\frac{\gamma^2}{\lambda}\mathcal{F}^{-1}\left[\frac{1}{\lambda^2+\kappa_2(\lambda)|\xi|^2}\xi^\top \xi(\widehat{f}_1(\xi)I+\widehat{f}_3(\xi))\right](x) \\
&\quad +\frac{\beta^2}{\lambda}\mathcal{F}^{-1}\left[\left(\frac{1}{\lambda^2+\kappa_1(\lambda)|\xi|^2}\widehat{\mathcal{P}}_s(\xi)+\frac{1}{\lambda^2+\kappa_2(\lambda)|\xi|^2}\widehat{\mathcal{P}}_c(\xi)\right)|\xi|^2\widehat{f}_3(\xi)\widehat{\mathcal{P}}_s(\xi)\right](x),
\end{aligned}
\end{align*}
respectively. We note that $\overline{R}_1(\lambda)\mathbf{f}$ and $\overline{R}_3(\lambda)\mathbf{f}$ are controlled by $\lambda^{-1}$ uniformly even if $\lambda$ is small. To overcome this problem, we assume  $\nabla f_1+\mathrm{div}{}^\top f_3=0$ and $f_3\in \nabla W^{1.p}(\mathbb{R}^3)^3$ to vanish these functions.

We are now in a position to study the case that $\lambda$ is not close to $0$. The existence and properties for solutions to \eqref{R-R3-1} under this case are summarized as follows.
\begin{prop}\label{R-R3-prop1}
Let $1<p<\infty$, $\mathbf{f}\in X^p(\mathbb{R}^3)$, $0<\varepsilon_0<\pi/2$ and $\lambda_0>0$. Then \eqref{R-R3-1} has a unique solution $\mathcal{R}_{\mathbb{R}^3}(\lambda)\mathbf{f}={}^\top(\Phi_{\mathbb{R}^3}(\lambda)\mathbf{f},\mathcal{W}_{\mathbb{R}^3}(\lambda)\mathbf{f},\mathcal{G}_{\mathbb{R}^3}(\lambda)\mathbf{f})\in Y^p(\mathbb{R}^3)$ satisfying $\mathcal{R}_{\mathbb{R}^3}(\lambda)\in \mathcal{A}(\Sigma_{\varepsilon_0}\cap \Xi,\mathcal{L}(X^p(\mathbb{R}^3),Y^p(\mathbb{R}^3)) )$ and
\begin{equation}\label{ineq1-R-R3-prop1}
|\lambda|\left\|\mathcal{R}_{\mathbb{R}^3}(\lambda)\mathbf{f}\right\|_{X^{p}(\mathbb{R}^3)}+\sum_{j=1}^2|\lambda|^{\frac{2-j}{2}}\left\|\nabla ^j\mathcal{W}_{\mathbb{R}^3}(\lambda)\mathbf{f}\right\|_{L^{p}(\mathbb{R}^3)}\leq C\left\|\mathbf{f}\right\|_{X^{p}(\mathbb{R}^3)}
\end{equation}
for any $\lambda\in V_{\varepsilon_0,\lambda_0}$ with some positive constant $C=C(p,\lambda_0,\varepsilon_0,\nu,\tilde{\nu},\beta,\gamma)$.
\end{prop}
In order to study \eqref{R-R3-phi}--\eqref{R-R3-G} for showing Proposition \ref{R-R3-prop1}, we focus on the following operators defined by 
\begin{align}
\mathfrak{L}_{l,\eta}(\lambda)f(x)&:=\mathcal{F}^{-1}\left[\frac{(i\xi)^\eta \widehat{f}(\xi)}{\lambda^2+\kappa_l(\lambda)|\xi|^2}\right](x),~
\mathfrak{M}_{l,\eta}(\lambda)f(x):=\mathcal{F}^{-1}\left[\frac{(i\xi)^\eta\widehat{f}(\xi)}{(\lambda^2+\kappa_l(\lambda)|\xi|^2)|\xi|^2}\right](x).
\label{pr-R-R3-prop1-2}
\end{align} 
Here $l=1,2$ and $\eta\in \mathbb{N}_0^3$. We then see that the structure of \eqref{R-R3-phi}--\eqref{R-R3-G} is based on $\lambda\mathfrak{L}_{l,0}(\lambda)$, $\mathfrak{L}_{l,\eta}(\lambda)$, $\lambda\mathfrak{M}_{l,\alpha }(\lambda)$ and $\mathfrak{M}_{l,\alpha+\eta}(\lambda)$, where $l$ and $\alpha,\eta\in\mathbb{N}_0^3$ satisfy $l=1,2$, $|\alpha|=2$ and $1\le|\eta|\le 2$.

We give the following lemma to control $\lambda^2+\kappa_l(\lambda)|\xi|^2~(l=1,2)$ appearing in \eqref{pr-R-R3-prop1-2}.
\begin{lem}\label{R-R3-lem1}
{\rm (i)} Let $\sigma_0>0$. Then there exists a positive constant $\sigma_1=\sigma_1(\nu,\tilde{\nu},\beta,\gamma,\sigma_0)$ such that the inequalities
\[
|\nu\lambda+\beta^2|\ge \sigma_1\left(|\lambda|+\frac{2\beta^2}{\nu}\right),~|(\nu+\tilde{\nu})\lambda+\beta^2+\gamma^2|\ge \sigma_1\left(|\lambda|+\frac{2(\beta^2+\gamma^2)}{\nu+\tilde{\nu}}\right)
\]
hold for any $\lambda\in\mathbb{C}$ satisfying $\min\{|\lambda+\beta^2/\nu|,|\lambda+(\beta^2+\gamma^2)/(\nu+\tilde{\nu})|\}\ge\sigma_0$.

\noindent
{\rm (ii)} Let $0<\varepsilon_0<\pi/2$. Then the inequality
\[
|\lambda+|\xi|^2|\ge \left(\sin\frac{\varepsilon_0}{2}\right)(|\lambda|+|\xi|^2)
\]
holds for any $\lambda\in \Sigma_{\varepsilon_0}$ and $\xi\in\mathbb{R}^3$.

\noindent
{\rm (iii)} Let $0<\varepsilon_0<\pi/2$ and $\lambda_0>0$. Then there exist positive constants $\sigma_2=\sigma_2(\nu,\tilde{\nu},\beta,\gamma,\varepsilon_0,\lambda_0)$ and $c_2=c_2(\nu,\tilde{\nu},\beta,\gamma,\varepsilon_0,\lambda_0)$ such that the inequalities
\begin{align*}
\max\left\{ \left| \arg\left( \frac{\lambda^2}{\kappa_1(\lambda)} \right) \right|, \left| \arg\left( \frac{\lambda^2}{\kappa_2(\lambda)} \right)  \right|  \right\} &\le \pi -\sigma_2,~
\min\left\{ \left|  \frac{\lambda^2}{\kappa_1(\lambda)} +|\xi|^2\right|, \left| \frac{\lambda^2}{\kappa_2(\lambda)} +|\xi|^2 \right| \right\}\ge c_2(|\lambda|+|\xi|^2)
\end{align*}
hold for any $\lambda\in V_{\varepsilon_0,\lambda_0}$ and $\xi\in\mathbb{R}^3$.

\noindent
{\rm (iv)}
Let $\lambda_0>0$ and $0<a<b$. Then the following inequality 
\[
(|\lambda|+|\xi|^2)^{-b}\le \lambda_0^{a-b}(|\lambda|+|\xi|^2)^{-a}
\]
holds for any $\lambda\in\mathbb{C}$ with $|\lambda|\ge\lambda_0$ and $\xi\in\mathbb{R}^3$.
\end{lem}
The proof of Lemma \ref{R-R3-lem1} can be verified in \cite{SE18, ShTa04}. Using Lemma \ref{R-R3-lem1} and Leibniz' rule, we have the following lemma which enables to apply Lemma \ref{FMthm} to \eqref{pr-R-R3-prop1-2}.
\begin{lem}\label{R-R3-lem2}
Let $0<\varepsilon_0<\pi/2$ and $\lambda_0>0$, and let $\eta\in\mathbb{N}_0^3$ be multi-indices with $|\eta|\le 2$. Then the following estimates 
\begin{align*}
&
{\rm (i)}\quad \left|\partial_\xi^\alpha\left(\frac{\xi^{\eta}}{|\xi|^2}\right)\right|\le C_\alpha|\xi|^{-2+|\eta|-|\alpha|},
\\
&{\rm (ii)}\quad \left|\partial_\xi^\alpha\left(\frac{\lambda}{\lambda^2+\kappa_l(\lambda)|\xi|^2}\right)\right|+\left|\partial_\xi^\alpha\left(\frac{\xi^\eta}{\lambda^2+\kappa_l(\lambda)|\xi|^2}\right)\right|\leq C_\alpha|\lambda|^{-1} |\xi|^{-|\alpha|}~(1\le |\eta|\le 2),\\
&{\rm (iii)}\quad \left|\partial_\xi^\alpha\left(\frac{\lambda \xi^{\eta}}{\lambda^2+\kappa_l(\lambda)|\xi|^2}\right)\right|\leq C_\alpha|\lambda|^{-\frac{1}{2}} |\xi|^{-|\alpha|}~(|\eta|=1)
\end{align*}
hold for any $\xi\in\mathbb{R}^3\setminus\{0\}$, $\alpha\in \mathbb{N}_0^3~(|\alpha|\le 2)$,~$l=1,2$ and $\lambda\in V_{\varepsilon_0,\lambda_0}$ with some constant $C_\alpha=C_\alpha(\nu,\tilde{\nu},\beta,\gamma,\varepsilon_0,\lambda_0)$.
\end{lem}
\begin{proof}[Proof of Proposition \ref{R-R3-prop1}]
The estimate \eqref{ineq1-R-R3-prop1} follows from \eqref{pr-R-R3-prop1-2} and Lemma \ref{R-R3-lem2}. This completes the proof of Proposition \ref{R-R3-prop1}.
\end{proof}

\vskip2mm

We next investigate the case that $\lambda$ is close to $0$.
To state the results, we employ the cut-off technique to separate \eqref{R-R3-phi}--\eqref{R-R3-G} into their low frequency part and high frequency part with respect to $\xi\in\mathbb{R}^3$. Let $\varphi_0,\varphi_\infty\in C^\infty(\mathbb{R}^3)$ be radial symmetric functions taken by $0\le \varphi_0(\xi)\le 1~(\xi\in\mathbb{R}^3),~\varphi_0(\xi)=
1~( |\xi|\le 1), ~
\varphi_0(\xi)=0~(|\xi|\ge2)$ 
and
$
\varphi_\infty(\xi):=1-\varphi_0(\xi).
$
It follows from \eqref{R-R3-phi}--\eqref{R-R3-G} with $\overline{R}_1(\lambda)\mathbf{f}=0$ and $\overline{R}_3(\lambda)\mathbf{f}=0$ that $\mathcal{R}_{\mathbb{R}^3}(\lambda)\mathbf{f}$ is decomposed as 
\[
\mathcal{R}_{\mathbb{R}^3}(\lambda)\mathbf{f}=\sum_{j=0,\infty}\mathcal{R}_{\mathbb{R}^3}^j(\lambda)\mathbf{f}, \quad \mathcal{R}_{\mathbb{R}^3}^j(\lambda)\mathbf{f}:={}^\top(\Phi_{\mathbb{R}^3}^j(\lambda)\mathbf{f},\mathcal{W}_{\mathbb{R}^3}^j(\lambda)\mathbf{f},\mathcal{G}_{\mathbb{R}^3}^j(\lambda)\mathbf{f}),
\]
where
\begin{align}
&\label{R-R3-phi-j}
(\Phi_{\mathbb{R}^3}^j(\lambda)\mathbf{f})(x):=\mathcal{F}^{-1}\left[\varphi_j(\xi)\frac{\lambda+(\nu+\tilde{\nu})|\xi|^2}{\lambda^2+\kappa_2(\lambda)|\xi|^2}\widehat{f}_1(\xi)\right](x) -i\mathcal{F}^{-1}\left[\varphi_j(\xi)\frac{\xi\cdot\widehat{f}_2(\xi)}{\lambda^2+\kappa_2(\lambda)|\xi|^2}\right](x),
\\
&\label{R-R3-w-j}
\begin{aligned}
(\mathcal{W}_{\mathbb{R}^3}^j(\lambda)\mathbf{f})(x) 
&:=-i\gamma^2\mathcal{F}^{-1}\left[\varphi_j(\xi)\frac{\widehat{f}_1(\xi)}{\lambda^2+\kappa_2(\lambda)|\xi|^2}\xi\right](x) \\
&\quad+\mathcal{F}^{-1}\left[\varphi_j(\xi)\left(\frac{\lambda}{\lambda^2+\kappa_1(\lambda)|\xi|^2}\widehat{\mathcal{P}}_s(\xi)+\frac{\lambda}{\lambda^2+\kappa_2(\lambda)|\xi|^2}\widehat{\mathcal{P}}_c(\xi)\right)\widehat{f}_2(\xi)\right](x)\\
&\qquad +i\beta^2\mathcal{F}^{-1}\left[\varphi_j(\xi)\left(\frac{1}{\lambda^2+\kappa_1(\lambda)|\xi|^2}\widehat{\mathcal{P}}_s(\xi)+\frac{1}{\lambda^2+\kappa_2(\lambda)|\xi|^2}\widehat{\mathcal{P}}_c(\xi)\right)\widehat{f}_3(\xi)\xi\right](x),
\end{aligned}
\\
&\label{R-R3-G-j}
\begin{aligned}
(\mathcal{G}_{\mathbb{R}^3}^j(\lambda)\mathbf{f})(x)
&:=i\mathcal{F}^{-1}\left[\varphi_j(\xi)\left(\frac{1}{\lambda^2+\kappa_1(\lambda)|\xi|^2}\widehat{\mathcal{P}}_s(\xi)+\frac{1}{\lambda^2+\kappa_2(\lambda)|\xi|^2}\widehat{\mathcal{P}}_c(\xi)\right)\widehat{f}_2(\xi)^\top\xi\right](x)\\
&\quad +\mathcal{F}^{-1}\left[\varphi_j(\xi)\left(\frac{\lambda+\nu|\xi|^2}{\lambda^2+\kappa_1(\lambda)|\xi|^2}\widehat{\mathcal{P}}_s(\xi)+\frac{\lambda+(\nu+\tilde{\nu})|\xi|^2}{\lambda^2+\kappa_2(\lambda)|\xi|^2}\widehat{\mathcal{P}}_c(\xi)\right)\widehat{f}_3(\xi)\right](x).
\end{aligned}
\end{align}
We also note that $\mathcal{G}_{\mathbb{R}^3}(\lambda)\mathbf{f}$ is expressed as $\mathcal{G}_{\mathbb{R}^3}(\lambda)\mathbf{f}=\nabla\Psi_{\mathbb{R}^3}(\lambda)\mathbf{f}$. Here $\Psi_{\mathbb{R}^3}(\lambda)\mathbf{f}$ is a $\mathbb{R}^3$ valued function given by $\Psi_{\mathbb{R}^3}(\lambda)\mathbf{f}=\sum_{j=0,\infty}\Psi_{\mathbb{R}^3}^j(\lambda)\mathbf{f}$, where
\begin{gather}
\label{R-R3-psi-j}
\begin{aligned}
(\Psi_{\mathbb{R}^3}^j (\lambda)\mathbf{f})(x)
&:=\mathcal{F}^{-1}\left[\varphi_j(\xi)\left(\frac{1}{\lambda^2+\kappa_1(\lambda)|\xi|^2}\widehat{\mathcal{P}}_s(\xi)+\frac{1}{\lambda^2+\kappa_2(\lambda)|\xi|^2}\widehat{\mathcal{P}}_c(\xi)\right)\widehat{f}_2(\xi)\right](x)\\
&\quad -i\mathcal{F}^{-1}\left[\varphi_j(\xi)\left(\frac{\lambda+\nu|\xi|^2}{\lambda^2+\kappa_1(\lambda)|\xi|^2}\widehat{\mathcal{P}}_s(\xi)+\frac{\lambda+(\nu+\tilde{\nu})|\xi|^2}{\lambda^2+\kappa_2(\lambda)|\xi|^2}\widehat{\mathcal{P}}_c(\xi)\right)\widehat{f}_3(\xi)\frac{\xi}{|\xi|^2}\right](x).
\end{aligned}
\end{gather}
Here we have used the fact that $f_3=-\nabla(-\Delta_{\mathbb{R}^3})^{-1}\mathrm{div}f_3$, provided that $f_3$ belongs to $\nabla W^{1,p}(\mathbb{R}^3)^3$. We state their properties in the following proposition for the case that $\lambda$ is close to $0$.
\begin{prop}\label{R-R3-prop2}
Let  $1<p<\infty$ and $R, R^\prime>0$, and let $\lambda_1=\lambda_1(\nu,\tilde{\nu},\beta,\gamma)>0$ be a positive constant such that $\max\{|\lambda^2/\kappa_1(\lambda)|, |\lambda^2/\kappa_2(\lambda)|\}\le 1/2$ holds for any $\lambda\in U_{\lambda_1}$. Then $\mathcal{R}_{\mathbb{R}^3}^0(\lambda)$, $\Psi_{\mathbb{R}^3}^0(\lambda)$, $\mathcal{R}_{\mathbb{R}^3}^\infty(\lambda)$ and $\Psi_{\mathbb{R}^3}^\infty(\lambda)$ satisfy the following assertions.

\vskip2mm

\noindent
{\rm (i) (low frequency part)} Let $\alpha\in\mathbb{N}_0^3$, $\lambda\in \Dot{U}_{\lambda_1}$ and $\mathbf{f}\in L_R^p(\mathbb{R}^3)$. Then the two functions $\partial_x^\alpha \mathcal{R}_{\mathbb{R}^3}^0(\lambda)\mathbf{f}$ and $\partial_x^\alpha \Psi_{\mathbb{R}^3}^0(\lambda)\mathbf{f}$ are written as the sum of the singular and  holomorpic parts:
\begin{align*}
&\partial_x^\alpha \mathcal{R}_{\mathbb{R}^3}^0(\lambda)\mathbf{f}=\widetilde{\mathcal{R}}_{\mathbb{R}^3}^{0,\alpha}(\lambda)\mathbf{f}+\mathcal{R}_{\mathbb{R}^3}^{0,c,\alpha}(\lambda)\mathbf{f}, \quad
\partial_x^\alpha \Psi_{\mathbb{R}^3}^0(\lambda)\mathbf{f}=\widetilde{\Psi}_{\mathbb{R}^3}^{0,\alpha}(\lambda)\mathbf{f}+\Psi_{\mathbb{R}^3}^{0,c,\alpha}(\lambda)\mathbf{f}. 
\end{align*}
Here $\widetilde{\mathcal{R}}_{\mathbb{R}^3}^{0,\alpha}(\lambda)$ and $\widetilde{\Psi}_{\mathbb{R}^3}^{0,\alpha}(\lambda)$ are operator-valued functions satisfying 
\begin{align*}
&\widetilde{\mathcal{R}}_{\mathbb{R}^3}^{0,\alpha}(\lambda)\in \mathcal{A}(\dot{U}_{\lambda_1}\cap \Xi, \mathcal{L}(X_R^{p}(\mathbb{R}^3),L^{p}(B_{R^\prime})\times L^{p}(B_{R^\prime})^3 \times L^{p}(B_{R^\prime})^{3\times3} ) ),\\
&\widetilde{\Psi}_{\mathbb{R}^3}^{0,\alpha}(\lambda)\in \mathcal{A}(\dot{U}_{\lambda_1}\cap \Xi, \mathcal{L}(X_R^{p}(\mathbb{R}^3),L^{p}(B_{R^\prime})^3 ) )
\end{align*}
and the estimates
\begin{align*}
&\left\|\frac{d^m }{d\lambda^m}\widetilde{\mathcal{R}}_{\mathbb{R}^3}^{0,\alpha}(\lambda)\mathbf{f} \right\|_{L^p(B_{R^\prime})}\leq C\|\mathbf{f}\|_{X^p(\mathbb{R}^3)}
\begin{cases}
|\lambda|^{2-m} \left|\log |\lambda|\right| & (m\le 2), \\
|\lambda|^{2-m} &(m\ge3),
\end{cases}  
\\
& \left\| \frac{d^m}{d\lambda^m}\widetilde{\Psi}_{\mathbb{R}^3}^{0,\alpha} (\lambda)\mathbf{f}\right\|_{L^p( B_{R^\prime})} 
\le C\|\mathbf{f}\|_{X^p(\mathbb{R}^3)}
\begin{cases}
|\lambda|^{1-m}\left|\log |\lambda|\right| & (m\le 1), \\
|\lambda|^{1-m} &(m\ge2)
\end{cases}
\end{align*}
for $m\in\mathbb{N}_0, \lambda\in \Dot{U}_{\lambda_1}$ and $\mathbf{f}\in L_R^p(\mathbb{R}^3)$, where $C=C_{\alpha,m,p,R,R^\prime,\nu,\tilde{\nu},\beta,\gamma}$ is a positive constant; $\mathcal{R}_{\mathbb{R}^3}^{0,c,\alpha}(\lambda)$ and $\Psi_{\mathbb{R}^3}^{0,c,\alpha}(\lambda)$ are operator functions satisfying 
\begin{align*}
&\mathcal{R}_{\mathbb{R}^3}^{0,c,\alpha}(\lambda)\in \mathcal{A}(U_{\lambda_1} \cap \Xi, \mathcal{L}(X_R^{p}(\mathbb{R}^3),L^{p}(B_{R^\prime})\times L^{p}(B_{R^\prime})^3 \times L^{p}(B_{R^\prime})^{3\times3} )),\\
&\Psi_{\mathbb{R}^3}^{0,c,\alpha}(\lambda)\in \mathcal{A}(U_{\lambda_1} \cap \Xi, \mathcal{L}(X_R^{p}(\mathbb{R}^3),L^{p}(B_{R^\prime})^3)).
\end{align*}

\vskip2mm
\noindent
{\rm (ii) (high frequency part)} The operator-valued functions $\mathcal{R}_{\mathbb{R}^3}^\infty(\lambda)$ and $\Psi_{\mathbb{R}^3}^\infty(\lambda)$ satisfy 
\begin{align*}
&\mathcal{R}_{\mathbb{R}^3}^\infty(\lambda)\in \mathcal{A}(U_{\lambda_1}, \mathcal{L}(X^p(\mathbb{R}^3),Y^p(\mathbb{R}^3))),
~\Psi_{\mathbb{R}^3}^\infty(\lambda)\in \mathcal{A}(U_{\lambda_1},\mathcal{L}(X^p(\mathbb{R}^3),W^{2,p}(\mathbb{R}^3)^3) ).
\end{align*}
\end{prop}

\vskip 1mm
Proposition \ref{R-R3-prop2} gives the following consequence for $\mathcal{R}_{\mathbb{R}^3}(\lambda)$ and $\Psi_{\mathbb{R}^3}(\lambda)$  with small $\lambda$.
\begin{cor}\label{R-R3-Cor1} 
Assume that the same assumptions in Proposition \ref{R-R3-prop2} hold. Then $\mathcal{R}_{\mathbb{R}^3}(\lambda)$ and $\Psi_{\mathbb{R}^3}(\lambda)$ satisfy
\begin{align*}
&\widetilde{\mathcal{R}}_{\mathbb{R}^3}^{0,\alpha}(\lambda)\in \mathcal{A}(\dot{U}_{\lambda_1}\cap \Xi, \mathcal{L}(X_R^{p}(\mathbb{R}^3),Y^p(B_{R^\prime})) ),
\quad
\widetilde{\Psi}_{\mathbb{R}^3}^{0,\alpha}(\lambda)\in \mathcal{A}(\dot{U}_{\lambda_1}\cap \Xi, \mathcal{L}(X_R^{p}(\mathbb{R}^3),L^{p}(B_{R^\prime})^3 ) )
\end{align*}
and the estimates
\begin{align}
&\left\| \frac{d^m}{d\lambda^m}\mathcal{R}_{\mathbb{R}^3}(\lambda)\mathbf{f}\right\|_{Y^p(B_{R^\prime})} 
\le C\|\mathbf{f}\|_{X^p(\mathbb{R}^3)}
\begin{cases}
1 & (m\le1), \\
\left|\log |\lambda|\right| & (m=2), \\
|\lambda|^{2-m} &(m\ge3),
\end{cases}
\label{R-R3-Cor1-1}\\
&\left\| \frac{d^m}{d\lambda^m}\Psi_{\mathbb{R}^3}(\lambda)\mathbf{f}\right\|_{W^{2,p}(B_{R^\prime})} 
\le C\|\mathbf{f}\|_{X^p(\mathbb{R}^3)}
\begin{cases}
1 & (m=0), \\
\left|\log |\lambda|\right|, & (m=1), \\
|\lambda|^{1-m} &(m\ge2)
\end{cases}
\label{R-R3-Cor1-2}
\end{align}
for $m\in\mathbb{N}_0$, $\lambda\in \Dot{U}_{\lambda_1}\cap \Xi$ and $f\in X_R^p(\mathbb{R}^3)$. Here $C=C_{m,p,R,R^\prime,\nu,\tilde{\nu},\beta,\gamma}$  is a positive constant. 
\end{cor}

\begin{rem}
Setting $\beta=0$ in \eqref{R-R3-w-j} formally, the coefficient of $\varphi_j(\xi)\widehat{\mathcal{P}}_s(\xi)\widehat{f}_2(\xi)$ in $\mathcal{F}\mathcal{W}_{\mathbb{R}^3}^j(\lambda)\mathbf{f}$ is given by $1/(\lambda+\nu|\xi|^2)$, which implies that the singular part of $\mathcal{W}_{\mathbb{R}^3}^0(\lambda)\mathbf{f}$ is only controlled by $\lambda^{1/2}$ $($See for {\rm \cite{K97,SE18}}$)$. This affects the slow decay of local energy of linearized semigroup in $\Omega$ when $\beta=0$.    
\end{rem}

\vskip 1mm

In order to show Proposition \ref{R-R3-prop2}, we investigate operators defined as
\begin{align}
\mathfrak{L}_{l,\eta}^j(\lambda)f(x)&:=\mathcal{F}^{-1}\left[\varphi_j(\xi)\frac{(i\xi)^\eta}{\lambda^2+\kappa_l(\lambda)|\xi|^2}\widehat{f}(\xi)\right](x),
\label{pr-R-R3-prop2-1}\\
\mathfrak{M}_{l,\eta}^j(\lambda)f(x)&:=\mathcal{F}^{-1}\left[\varphi_j(\xi)\frac{(i\xi)^\eta}{(\lambda^2+\kappa_l(\lambda)|\xi|^2)|\xi|^2}\widehat{f}(\xi)\right](x), 
\label{pr-R-R3-prop2-2}\\
\mathfrak{N}_{l,\eta}^j(\lambda)f(x)&:=\mathcal{F}^{-1}\left[\varphi_j(\xi)\frac{(i\xi)^\eta}{(\lambda^2+\kappa_l(\lambda)|\xi|^2)|\xi|^4}\widehat{f}(\xi)\right](x).
\label{pr-R-R3-prop2-3}
\end{align} 
Here $j=0,\infty$, $l=1,2$ and $\eta\in \mathbb{N}_0^3$. We then confirm that the properties of \eqref{R-R3-phi-j}--\eqref{R-R3-G-j} follow from $\lambda\mathfrak{L}_{l,0}^j(\lambda)$, $\mathfrak{L}_{l,\eta}^j(\lambda)$, $\lambda\mathfrak{M}_{l,\widetilde{\eta} }^j(\lambda)$ and $\mathfrak{M}_{l,\widetilde{\eta}+\eta}^j(\lambda)$, where $l$ and $\eta,\widetilde{\eta}\in\mathbb{N}_0^3$ satisfy $l=1,2$, $1\le|\eta|\le 2$ and $|\widetilde{\eta}|=2$. Similarly, the property of \eqref{R-R3-psi-j} is derived from $\mathfrak{L}_{l,0}^j(\lambda)$, $\mathfrak{L}_{l,\eta}^j(\lambda)$, $\lambda\mathfrak{M}_{l,\eta }^j(\lambda)$, $\mathfrak{M}_{l,\tilde{\eta} }^j(\lambda)$, $\mathfrak{M}_{l,\widetilde{\eta}+\eta}^j(\lambda)$ and $\lambda\mathfrak{N}_{l,\widetilde{\eta}+\eta}^j(\lambda)$, where $l$ and $\eta,\widetilde{\eta}\in\mathbb{N}_0^3$ satisfy $l=1,2$, $|\eta|= 1$ and $|\widetilde{\eta}|=2$. Thus, it is enough to study   $\mathfrak{L}_{l,\eta_1}^j(\lambda)$, $\mathfrak{M}_{l,\eta_2}^j(\lambda)$ and $\mathfrak{N}_{l,\eta_3}^j(\lambda)$ with $|\eta_1|\le 2,~1\le|\eta_2|\le 4$ and $|\eta_3|=3$ to show Proposition \ref{R-R3-prop2}. 
In order to prove Proposition \ref{R-R3-prop2} (i), we employ the following key lemma to reformulate \eqref{pr-R-R3-prop2-1}--\eqref{pr-R-R3-prop2-3} with $j=0$.
\begin{lem}[\cite{SE18}]\label{R-R3-lem6}
Let $z\in \Dot{U}_{1/2}$ and $N\in\mathbb{N}_0\cup \{-1\}$. Then we have the following identities 
\begin{align*}
&\int_0^2 \varphi_0(r)\frac{r^{2N+2}}{z+r^2}dr= \frac{\pi}{2}(-1)^{N+1}z^{N+\frac{1}{2}}+ h_{2N}(z), \\
&\int_0^2 \varphi_0(r)\frac{r^{2N+3}}{z+r^2}dr= -\frac{1}{2}(-z)^{N+1}\log z+ h_{2N+1}(z), 
\end{align*}
where $h_{2N}(z)$ and $h_{2N+1}(z)$ are holmorphic functions in $z\in U_{1/2}$, given by
\begin{align*}
&
\begin{aligned}
h_{2N}(z)&:=\int_1^2 \varphi_0(r)\frac{r^{2N+2}}{z+r^2}dr-\frac{1}{2}(-z)^{N+1}\int_{|r|\ge1}\frac{dr}{z+r^2} \\
&\quad +\sum_{l=1}^{N+1} \binom{N+1}{l} (-z)^{N+1-l}\int_0^1(z+r^2)^{l-1}dr \quad (N\in \mathbb{N}_0), 
\end{aligned}
\\
&
\begin{aligned}
h_{2N+1}(z)&:=\int_1^2 \varphi_0(r)\frac{r^{2N+3}}{z+r^2}dr +\frac{1}{2}(-z)^{N+1}\log(1+z)\\
&\quad+\sum_{l=1}^{N+1}\binom{N+1}{l}(-z)^{N+1-l}\int_0^1(z+s)^{l-1}ds \quad (N\in \mathbb{N}_0),
\end{aligned}
\\
&h_{-1}(z):=\int_1^2 \varphi_0(r)\frac{r}{z+r^2}dr+\frac{1}{2}\log(1+z), 
\quad h_{-2}(z):=\int_1^2 \frac{\varphi_0(r)}{z+r^2}dr-\frac{1}{2}\int_{|r|\ge1}\frac{dr}{z+r^2}.
\end{align*}
\end{lem}

\begin{proof}[Proof of Proposition \ref{R-R3-prop2} {\rm (i)}]
Let $\mathfrak{K}(\lambda)$ be an operator-valued function belonging to 
$$
\bigcup_{\substack{\eta_1\in \mathbb{N}_0^3 \\ (|\eta_1|\le2)} }\{\mathfrak{L}_{1,\eta_1}^0(\lambda),\mathfrak{L}_{2,\eta_1}^0(\lambda)\} \cup \bigcup_{\substack{\eta_2\in \mathbb{N}_0^3 \\ (1\le |\eta_2|\le4)} }\{\mathfrak{M}_{1,\eta_2}^0(\lambda),\mathfrak{M}_{2,\eta_2}^0(\lambda)\} \cup \bigcup_{\substack{\eta_3\in \mathbb{N}_0^3 \\ (|\eta_3|=3)} }\{\mathfrak{N}_{1,\eta_3}^0(\lambda), \mathfrak{N}_{2,\eta_3}^0(\lambda)\}.
$$
We are going to show that for any $\alpha\in\mathbb{N}_0^3$, $\lambda\in\Dot{U}_{\lambda_1}$ and $f\in L_R^p(\mathbb{R}^3)$, $\partial_x^{\alpha} \mathfrak{K}(\lambda)f(x)$ is decomposed as
\begin{align}\label{pr-R-R3-prop2-4}
\partial_x^{\alpha} \mathfrak{K}(\lambda)f(x)=\widetilde{\mathfrak{K}}^\alpha(\lambda)f(x)+\mathfrak{K}_{\mathrm{hol}}^{\alpha}(\lambda)f(x).
\end{align}
Here $\widetilde{\mathfrak{K}}^\alpha(\lambda)$ and $\mathfrak{K}_{\mathrm{hol}}^{\alpha}(\lambda)$ are two operator-valued functions satisfying 
\begin{gather}\label{pr-R-R3-prop2-4-1} 
\widetilde{\mathfrak{K}}^\alpha(\lambda)\in \mathcal{A}(\dot{U}_{\lambda_1}, \mathcal{L}(L_R^{p}(\mathbb{R}^3), L^p(B_{R^\prime}) ) ),~\mathfrak{K}_{\mathrm{hol}}^{\alpha}(\lambda)\in \mathcal{A}(U_{\lambda_1}, \mathcal{L}(L_R^{p}(\mathbb{R}^3), L^p(B_{R^\prime}) ) ) 
\end{gather} 
and the following three estimates for any $m\in\mathbb{N}_0$, $\lambda\in \Dot{U}_{\lambda_1}\cap \Xi$ and $f\in L_R^p(\mathbb{R}^3)$:
\begin{align} 
&\left\|\frac{d^m }{d\lambda^m}\widetilde{\mathfrak{K}}^\alpha(\lambda)f \right\|_{L^p(B_{R^\prime})}
\leq C\|f\|_{L^p(\mathbb{R}^3)} 
|\lambda|^{1-m}
\label{pr-R-R3-prop2-5} 
\end{align}
when $\mathfrak{K}(\lambda)=\mathfrak{L}_{l,\eta_1}^0(\lambda)~(|\eta_1|=0,~l=1,2)$ or $\mathfrak{K}(\lambda)= \mathfrak{M}_{l,\eta_2}^0(\lambda)~(|\eta_2|=2,~l=1,2)$;
\begin{align}
&\left\|\frac{d^m }{d\lambda^m}\widetilde{\mathfrak{K}}^\alpha(\lambda)f \right\|_{L^p(B_{R^\prime})}
\leq C\|f\|_{L^p(\mathbb{R}^3)} 
\begin{cases}
|\lambda|^{2-m}\left|\log |\lambda|\right| & (m=0), \\
|\lambda|^{2-m} &(m\ge1)
\end{cases}  
\label{pr-R-R3-prop2-5-1} 
\end{align}
when $\mathfrak{K}(\lambda)=\mathfrak{L}_{l,\eta_1}^0(\lambda)~(1\le|\eta_1|\le2,~l=1,2)$ or $\mathfrak{K}(\lambda)= \mathfrak{M}_{l,\eta_2}^0(\lambda)~(3\le |\eta_2|\le 4,~l=1,2)$;
\begin{align}
&\left\|\frac{d^m }{d\lambda^m}\widetilde{\mathfrak{K}}^\alpha(\lambda)f \right\|_{L^p(B_{R^\prime})}
\leq C\|f\|_{L^p(\mathbb{R}^3)} 
\begin{cases}
\left|\log |\lambda|\right| & (m=0), \\
|\lambda|^{1-m} &(m\ge1)
\end{cases}  
\label{pr-R-R3-prop2-5-2} 
\end{align} 
when $\mathfrak{K}(\lambda)=\mathfrak{M}_{l,\eta_2}^0(\lambda)~(|\eta_2|=1,~l=1,2)$ or $\mathfrak{K}(\lambda)= \mathfrak{N}_{l,\eta_3}^0(\lambda)~(|\eta_3|= 3,~l=1,2)$, 
where $C=C_{\alpha,m,p,R,R^\prime,\nu,\tilde{\nu},\beta,\gamma}$ is a positive constant.  Using \eqref{pr-R-R3-prop2-4} with \eqref{pr-R-R3-prop2-4-1}--\eqref{pr-R-R3-prop2-5-2}, we earn Proposition \ref{R-R3-prop2} (i). 

For simplicity, we only consider the case $\mathfrak{K}(\lambda)=\mathfrak{M}_{l,\eta_2}^0(\lambda)~(|\eta_2|=1,l=1,2)$. The other cases can be proved in a similar manner. For given function $g: S\times\mathbb{R}^3\to \mathbb{R}$ and $N\in\mathbb{N}_0$, we put 
\[
J_N(g)(x):=\frac{1}{(2\pi)^3}\int_{\mathbb{R}^3} \int_{S^2} \{i(x-y)\cdot \omega\}^N g(\omega,y)dyd\omega~(x\in \mathbb{R}^3).
\] 
 Here $S^2$ is an unit sphere $S^2:=\{x\in\mathbb{R}^3~|~|x|=1\}$.
Taking the operator $\partial_x^{\alpha}$ to \eqref{pr-R-R3-prop2-2} and then applying the polar coordinate transform $x=r\omega~(r>0,~\omega\in S^2)$ and Taylor expansion $e^{ir\omega\cdot (x-y)}=\sum_{N=0}^\infty \{i r \omega\cdot (x-y)\}^N/(N !)$, $\partial_x^{\alpha} \mathfrak{M}_{l,\eta_2}^0(\lambda)f(x)$ is represented as
\begin{align*}
\partial_x^{\alpha} \mathfrak{M}_{l,\eta_2}^0(\lambda)f(x)&=\frac{1}{\kappa_l(\lambda)}\sum_{N=0}^{\infty} \frac{J_N((i\omega)^{\alpha+\eta_2} f)(x)}{N !} \int_0^2 \varphi_0(r) \frac{r^{|\alpha|+N +1 } }{r^2+\lambda^2\kappa_l(\lambda)^{-1} } dr
\end{align*}
for any  $\alpha\in (\mathbb{N}_0)^3$.
We then use Lemma \ref{R-R3-lem6} to the resultant expression to derive
\begin{equation*}
\partial_x^{\alpha} \mathfrak{M}_{l,\eta_2}^0(\lambda) f(x)  = \widetilde{\mathfrak{M} }_{l,\eta_2}^{0,\alpha}(\lambda) f(x) +\mathfrak{M}_{l,\eta_2}^{0,c,\alpha}(\lambda) f(x),
\end{equation*}
where 
\begin{align*}
\widetilde{\mathfrak{M} }_{l,\eta_2}^{0,\alpha}(\lambda)f(x)&:=-\frac{1}{2\kappa_l(\lambda)}\log \left(\frac{\lambda^2}{\kappa_l(\lambda)} \right)\sum_{N+|\alpha|: {\rm even}}  \left(-\frac{\lambda^2}{\kappa_l(\lambda)} \right)^{\frac{N+|\alpha|}{2} }\frac{J_N( (i\omega)^{\alpha+\eta_2}  f)(x)}{N !} \\
&\qquad -\frac{\pi}{2\kappa_l(\lambda)}\left(\frac{\lambda^2}{\kappa_l(\lambda)} \right)^{\frac{1}{2} }\sum_{N+|\alpha|: {\rm odd}}  \left(-\frac{\lambda^2}{\kappa_l(\lambda)} \right)^{\frac{N+|\alpha|-1}{2} }\frac{J_N( (i\omega)^{\alpha+\eta_2}  f)(x)}{N !}, \\
\mathfrak{M}_{l,\eta_2}^{0,c,\alpha}(\lambda)f(x)&:=\frac{1}{\kappa_l(\lambda)}\sum_{N=0}^\infty h_{N+|\alpha|-1}\left(\frac{\lambda^2}{\kappa_l(\lambda)} \right)\frac{J_N( (i\omega)^{\alpha+\eta_2}  f)(x)}{N !}.
\end{align*}
We next show 
\begin{equation}\label{pr-R-R3-prop2-7}
\widetilde{\mathfrak{M} }_{l,\eta_2}^{0,\alpha}(\lambda)\in \mathcal{A}(\Dot{U}_{\lambda_1},\mathcal{L}(L_R^p(\mathbb{R}^3), L^p(B_{R^\prime}))),~\mathfrak{M}_{l,\eta_2}^{0,c,\alpha}(\lambda)\in \mathcal{A}(U_{\lambda_1},\mathcal{L}(L_R^p(\mathbb{R}^3), L^p(B_{R^\prime}))).
\end{equation}
Since $f$ is zero for $x\in \mathbb{R}^3\setminus B_R$, we have
\[
\left| \left(-\frac{\lambda^2}{\kappa_l(\lambda)} \right)^{\frac{N+|\alpha|}{2} }\frac{J_N( (i\omega)^{\alpha+\eta_2}  f)(x)}{N !}\right|\leq C\frac{M^N}{N!}\|f\|_{L^p(\mathbb{R}^3)}
\]
when $N+|\alpha|$ is even. Here $C$ and $M$ are positive constants depending only on $\alpha,~p,~R,~\nu,~\tilde{\nu},~\beta$ and $\gamma$. Similarly we obtain
\[
\left| \left(-\frac{\lambda^2}{\kappa_l(\lambda)} \right)^{\frac{N+|\alpha|-1}{2} }\frac{J_N( (i\omega)^{\alpha+\eta_2}  f)(x)}{N !}\right|\leq C\frac{M^N}{N!}\|f\|_{L^p(\mathbb{R}^3)}
\]
when $N+|\alpha|$ is odd.  Here $C$ and $M$ are positive constants depending only on $\alpha,~p,~R,~\nu,~\tilde{\nu},~\beta$ and $\gamma$. Hence we see that $\widetilde{\mathfrak{M} }_{l,\eta_2}^{0,\alpha}(\lambda)\in \mathcal{A}(\Dot{U}_{\lambda_1},\mathcal{L}(L_R^p(\mathbb{R}^3), L^p(B_{R^\prime})))$ holds. We finally study the remaining operator $\mathfrak{M}_{l,\eta_2}^{0,c,\alpha}(\lambda)$. Since $h_{N+|\alpha|-1}(\lambda^2/\kappa_l(\lambda))$ is controlled as $|h_{N+|\alpha|-1}(\lambda^2/\kappa_l(\lambda))|\leq  C M^N$ with some positive constants $C$ and $M$ depending only on $\alpha,~p,~R,~\nu,~\tilde{\nu},~\beta$ and $\gamma$, we have
\[
\left| h_{N+|\alpha|-1}\left(\frac{\lambda^2}{\kappa_l(\lambda)} \right)\frac{J_N( (i\omega)^{\alpha+\eta_2}  f)(x)}{N !}\right|\leq C\frac{M^N}{N!}\|f\|_{L^p(\mathbb{R}^3)}.
\]
Therefore we arrive at \eqref{pr-R-R3-prop2-7}. As a result, setting $\widetilde{\mathcal{K}}^\alpha(\lambda):=\widetilde{\mathfrak{M} }_{l,\eta_2}^{0,\alpha}(\lambda)$ and $\mathcal{K}_{\mathrm{hol}}^\alpha(\lambda):=\mathfrak{M}_{l,\eta_2}^{0,c,\alpha}(\lambda)$, we arrive at the desired decomposition \eqref{pr-R-R3-prop2-4} with properties \eqref{pr-R-R3-prop2-4-1} and \eqref{pr-R-R3-prop2-5-2}. Here the estimate \eqref{pr-R-R3-prop2-5-2} follows from Leibniz rule and Bell's formula. This completes the proof of Proposition \ref{R-R3-prop2} (i).
\end{proof}
We next consider the high frequency part $\mathcal{R}_{\mathbb{R}^3}^\infty(\lambda)f$.  For $\eta\in \mathbb{N}_0^3$ with $|\eta|\le 2$, $\xi\in\mathbb{R}^3$ with $|\xi|\ge1$ and $\lambda\in U_{\lambda_1}$, $1/(\lambda^2+\kappa_l(\lambda)|\xi|^2)$ is expressed as 
\begin{gather*}
\frac{1}{\lambda^2+\kappa_l(\lambda)|\xi|^2}=\frac{1}{\kappa_l(\lambda)}\sum_{N=0}^\infty\left(-\frac{\lambda^2}{\kappa_l(\lambda)} \right)^N\frac{1}{|\xi|^{2N+2}},
\end{gather*}
provided that $|\lambda^2/(\kappa_l(\lambda)|\xi|^2)|\le 1/2$ holds.
In view of this expansion and Lemma \ref{FMthm}, the following lemma plays a key role. 
\begin{lem}\label{R-R3-lem4}
Let $l=1,2$ and let $\eta\in\mathbb{N}_0^3$ be a multi-index with $|\eta|\le 2$. Then the following estimates 
\begin{align*}
{\rm (i)}&\quad \left|\frac{\lambda^2}{\kappa_l(\lambda)}+|\xi|^2\right|\ge \frac{|\xi|^2}{2}, \\
{\rm (ii)}&\quad \left|\partial_\xi^\alpha\left(\frac{\varphi_\infty(\xi) \xi^\eta}{|\xi|^{2N+2}}\right)\right|\leq C_\alpha \left( \frac{1}{2}\right)^N (N+1)^2 |\xi|^{-|\alpha|}
\end{align*}
hold for any $\xi\in\mathbb{R}^3\setminus\{0\}$, $\alpha\in\mathbb{N}_0^3~(|\alpha|\le 2)$,~$N\in \mathbb{N}_0$ and $\lambda\in U_{\lambda_1}$ with some constant $C_\alpha=C_\alpha(\nu,\tilde{\nu},\beta,\gamma)$.
\end{lem}
The proof of Proposition \ref{R-R3-prop2} {\rm (ii)} is similar to \cite{SE18} by using Lemma \ref{FMthm} and Lemma \ref{R-R3-lem4}. We omit this detail.

\vskip1mm
We next deal with the case $\lambda=0$. Letting $\lambda=0$ in \eqref{R-R3-phi}--\eqref{R-R3-G} with $\overline{R}_1(\lambda)\mathbf{f}=0$ and $\overline{R}_3(\lambda)\mathbf{f}=0$ and \eqref{R-R3-psi-j}, the solution formula of \eqref{R-R3-1} with $\lambda=0$ and $\Psi_{\mathbb{R}^3}(0)\mathbf{f}$  are given by
\begin{align}
&\label{R-R3-0-phi}
(\Phi_{\mathbb{R}^3}(0)\mathbf{f})(x)=\frac{\nu+\tilde{\nu} }{\beta^2+\gamma^2}f_1(x) 
-\frac{i}{\beta^2+\gamma^2}\mathcal{F}^{-1}\left[\frac{\xi\cdot\widehat{f}_2(\xi)}{|\xi|^2}\right](x),
\\
&\label{R-R3-0-w}
\begin{aligned}
(\mathcal{W}_{\mathbb{R}^3}(0)\mathbf{f})(x) &=-\frac{i\gamma^2}{\beta^2+\gamma^2}\mathcal{F}^{-1}\left[\frac{\widehat{f}_1(\xi)}{|\xi|^2}\xi\right](x) \\
&\quad 
+i\mathcal{F}^{-1}\left[\widehat{\mathcal{P}}_s(\xi)\widehat{f}_3(\xi)\frac{\xi}{|\xi|^2}\right](x)+\frac{i\beta^2}{\beta^2+\gamma^2}\mathcal{F}^{-1}\left[\widehat{\mathcal{P}}_c(\xi)\widehat{f}_3(\xi)\frac{\xi}{|\xi|^2}\right](x),
\end{aligned}
\\
&\label{R-R3-0-G}
\begin{aligned}
(\mathcal{G}_{\mathbb{R}^3}(0)\mathbf{f})(x) 
&=\frac{i}{\beta^2}\mathcal{F}^{-1}\left[\widehat{\mathcal{P}}_s(\xi)\widehat{f}_2(\xi)\frac{{}^\top \xi}{|\xi|^2}\right](x)+\frac{i}{\beta^2+\gamma^2}\mathcal{F}^{-1}\left[\widehat{\mathcal{P}}_c(\xi)\widehat{f}_2(\xi)\frac{{}^\top \xi}{|\xi|^2}\right](x)\\
&\quad
 +\frac{\nu}{\beta^2}\mathcal{F}^{-1}[\widehat{\mathcal{P}}_s(\xi)\widehat{f}_3(\xi)](x)+\frac{\nu+\tilde{\nu}}{\beta^2+\gamma^2}\mathcal{F}^{-1}[\widehat{\mathcal{P}}_c(\xi)\widehat{f}_3(\xi)](x),
\end{aligned}
\\
&\label{R-R3-0-psi}
\begin{aligned}
(\Psi_{\mathbb{R}^3}(0)\mathbf{f})(x) 
&=\frac{1}{\beta^2}\mathcal{F}^{-1}\left[\frac{1}{|\xi|^2}\widehat{\mathcal{P}}_s(\xi)\widehat{f}_2(\xi)\right](x)+\frac{1}{\beta^2+\gamma^2}\mathcal{F}^{-1}\left[\frac{1}{|\xi|^2}\widehat{\mathcal{P}}_c(\xi)\widehat{f}_2(\xi)\right](x)\\[1ex]
&\quad 
-\frac{i\nu}{\beta^2}\mathcal{F}^{-1}\left[\widehat{\mathcal{P}}_s(\xi)\widehat{f}_3(\xi)\frac{\xi}{|\xi|^2}\right](x)-\frac{i(\nu+\tilde{\nu})}{\beta^2+\gamma^2}\mathcal{F}^{-1}\left[\widehat{\mathcal{P}}_c(\xi)\widehat{f}_3(\xi)\frac{\xi}{|\xi|^2}\right](x).
\end{aligned}
\end{align}
The details of \eqref{R-R3-0-phi}--\eqref{R-R3-0-psi} are summarized as follows.
\begin{prop}\label{R-R3-prop3}
Let $1<p<\infty$, $R,~R^\prime>0$ and $\mathbf{f}\in X_R^p(\mathbb{R}^3)$. Then $\mathcal{R}_{\mathbb{R}^3}(0)$ and $\Psi_{\mathbb{R}^3}(0)$ satisfies the following properties:
\begin{align}
&{\rm (i)}~\mathcal{R}_{\mathbb{R}^3}(0)\in \mathcal{L}(X_R^p(\mathbb{R}^3),Y^p(B_{R^\prime})),~\Psi_{\mathbb{R}^3}(0)\in \mathcal{L}(X_R^p(\mathbb{R}^3),W^{2,p}(B_{R^\prime})^3), \label{R-R3-prop3-1} \\
&{\rm (ii)}~ 
\left\|\nabla \Phi_{\mathbb{R}^3}(0)\mathbf{f}\right\|_{L^p(\mathbb{R}^3)}+\left\|\nabla \mathcal{W}_{\mathbb{R}^3}(0)\mathbf{f}\right\|_{W^{1,p}(\mathbb{R}^3)}+\left\|\nabla \mathcal{G}_{\mathbb{R}^3}(0)\mathbf{f}\right\|_{L^p(\mathbb{R}^3)}
\leq 
C\|\mathbf{f}\|_{X^p(\mathbb{R}^3)}, 
\label{R-R3-prop3-2}
\\
&{\rm (iii)}~
 |x|^2|(\mathcal{R}_{\mathbb{R}^3}(0)\mathbf{f})(x)|+|x||(\Psi_{\mathbb{R}^3}(0)\mathbf{f})(x)|\leq C \|\mathbf{f}\|_{X^p(\mathbb{R}^3)}~(|x|\ge 2R), 
\label{R-R3-prop3-3} \\
&{\rm (iv)}~
\lim_{|\lambda|\to 0,~|\arg \lambda|\le \frac{\pi}{2}}(\|(\mathcal{R}_{\mathbb{R}^3}(\lambda)-\mathcal{R}_{\mathbb{R}^3}(0))\mathbf{f}\|_{Y^p(B_{R^\prime})}+\|(\Psi_{\mathbb{R}^3}(\lambda)-\Psi_{\mathbb{R}^3}(0))\mathbf{f}\|_{W^{2,p}(B_{R^\prime})})=0. \label{R-R3-prop3-4}
\end{align}
\end{prop}
Proposition \ref{R-R3-prop3} is established by Lemma \ref{R-R3-lem6}, Lemma \ref{R-R3-lem4} and the following lemma.
\begin{lem}[\cite{SS01}, Theorem 2.3.]\label{SS01-thm2.3.}
Let $M\in\mathbb{N}_0$, $0<\sigma\le 1$ and $\delta:=M+\sigma-3$. If a complex-valued function $f\in C^\infty(\mathbb{R}^3\setminus\{0\})$ satisfies $\partial_\xi^\alpha f\in L^1(\mathbb{R}^3)~(|\alpha|\le M)$ and
\[
|\partial_\xi^\alpha f(\xi)|\le C_\alpha |\xi|^{\delta-|\alpha|}~(\xi\in\mathbb{R}^3\setminus\{0\},~\alpha\in (\mathbb{N}_0)^3),
\] 
then the following estimate hold:
\[
|\mathcal{F}^{-1}[f](x)|\le C_\delta (\max_{|\alpha|\le M+2}C_\alpha) |x|^{-3-\delta}~(x\in\mathbb{R}^3\setminus\{0\}).
\]
\end{lem}
\begin{proof}[Proof of Proposition \ref{R-R3-prop3}] 
The properties \eqref{R-R3-prop3-1}, \eqref{R-R3-prop3-2} and \eqref{R-R3-prop3-4} is shown in a similar way to \cite{SE18} and Proposition \ref{R-R3-prop2}. Therefore we omit the proof of Proposition \ref{R-R3-prop3} (i), (ii) and (iv), and concentrate to prove Proposition \ref{R-R3-prop3} (iii). As in the proof of Proposition \ref{R-R3-prop1} and Proposition \ref{R-R3-prop2}, it is enough to deal with operators
\begin{align*}
[\mathcal{M}_\eta f](x)&:=\mathcal{F}^{-1}\left[\frac{\xi^\eta}{|\xi|^2}\widehat{f}(\xi)\right](x),~
[\mathcal{M}_\eta^j f](x):=\mathcal{F}^{-1}\left[\frac{\xi^\eta}{|\xi|^2}\varphi_j(\xi)\widehat{f}(\xi)\right](x)~(j=0,\infty,~\eta\in\mathbb{N}_0^3,~|\eta|\le 2),  \\
[\mathcal{N}_\eta f](x)&:=\mathcal{F}^{-1}\left[\frac{\xi^{\eta}}{|\xi|^4}\widehat{f}(\xi)\right](x),~ 
[\mathcal{N}_\eta^j f](x):=\mathcal{F}^{-1}\left[\frac{\xi^{\eta} }{|\xi|^4}\varphi_j(\xi)\widehat{f}(\xi)\right](x)~(j=0,\infty,~\eta\in\mathbb{N}_0^3,~2\le|\eta|\le 3). 
\end{align*}
We first look at $\mathcal{M}_\eta^0 f$ with $f\in L_R^p(\mathbb{R}^3)$. Since $f\in L^1(\mathbb{R}^3)$ is satisfied, $(\xi^\eta/|\xi|^2)\varphi_0(\xi)\widehat{f}(\xi)$ belongs to $C^\infty(\mathbb{R}^3\setminus\{0\})$ and $L^1(\mathbb{R}^3)$ for $\eta\in\mathbb{N}_0^3$ with $0\le |\eta|\le 2$. Moreover by virtue of Leibniz' rule and Lemma \ref{R-R3-lem2} (i), $(\xi^\eta/|\xi|^2)\varphi_0(\xi)\widehat{f}(\xi)$ satisfies
\begin{equation*}
\left|\partial_\xi^\alpha \left(\frac{\xi^\eta}{|\xi|^2}\varphi_0(\xi)\widehat{f}(\xi)\right)\right|\le C_{\alpha,\eta,p,R} |\xi|^{|\eta|-2-|\alpha|}\|f\|_{L^p(\mathbb{R}^3)}
\end{equation*}
for $\xi\in\mathbb{R}^3\setminus\{0\}$ and $\alpha\in\mathbb{N}_0^3$. Hence applying Lemma \ref{SS01-thm2.3.} with $M=|\eta|$, $\sigma=1$, $\delta=M+\sigma-3=|\eta|-2$, $f(\xi)=(\xi^\eta/|\xi|^2)\varphi_0(\xi)\widehat{f}(\xi)$, $\alpha=\eta$ and $C_\alpha=C_{\alpha,\eta,p,R}$, we obtain
\begin{equation}\label{pr-R-R3-prop3-1}
\left| {\mathcal{M}_\eta^0 f}(x) \right|\le C_{p,R}|x|^{-1-|\eta|}\|f\|_{L^p(\mathbb{R}^3)}~(x\in\mathbb{R}^3\setminus\{0\},~0\le|\eta|\le2).  
\end{equation}
We next investigate $\mathcal{M}_{\eta}^\infty f$. 
Let $K_\eta(x)$ be
\[
K_\eta(x) :=\mathcal{F}^{-1}\left[\frac{\xi^\eta}{|\xi|^2}\varphi_\infty(\xi) \right](x).
\]
Estimating
\[
\|x^\alpha K_\eta\|_{L^\infty(\mathbb{R}^3)}\le C\left\|\frac{\partial^\alpha}{\partial \xi^\alpha} \left(\frac{\xi^\eta}{|\xi|^2}\varphi_\infty(\xi) \right) \right\|_{L^1(\mathbb{R}^3\setminus B_1)}
\]
for $\alpha,\eta\in \mathbb{N}_0^3$ and assuming $|\alpha|>|\eta|+1$ and $|x|\ge R$, we arrive at
$
|K_\eta(x)|\le C_\eta |x|^{-1-|\eta|}, 
$ 
which gives the following estimate for $[\mathcal{M}_\eta^\infty f](x)$ with $|x|\ge 2R$:
\begin{gather}\label{pr-R-R3-prop3-2}
\begin{aligned}
|[\mathcal{M}_\eta^\infty f](x)|
= |(K_\eta * f)(x)| 
&= \left|\int_{B_R} K_\eta(x-y) f(y) dy \right| \\
&\le C_\eta \int_{B_R} \frac{|f(y)|}{|x-y|^{1+|\eta|}} dy \\
&\le C_{\eta,R}|x|^{-1-|\eta|}\|f\|_{L^1(B_R)}\le C_{\eta,p,R}|x|^{-1-|\eta|}\|f\|_{L^p(\mathbb{R}^3)}.
\end{aligned}
\end{gather}
As a result, combining \eqref{pr-R-R3-prop3-1} and \eqref{pr-R-R3-prop3-2} leads to
\begin{gather}\label{pr-R-R3-prop3-3}
|[\mathcal{M}_\eta f](x)|\le C_{\eta,p,R}|x|^{-1-|\eta|}\|f\|_{L^p(\mathbb{R}^3)}~(|x|\ge 2R).
\end{gather}
Similarly, $[\mathcal{N}_{\eta} f](x)$ with $|x|\ge 2R$, $f\in L_R^p(\mathbb{R}^3)$ and $2\le |\eta|\le 3$ is controlled as  
\begin{gather}\label{pr-R-R3-prop3-4}
|[\mathcal{N}_{\eta} f](x)|\le C_{\eta,p,R}|x|^{1-|\eta|}\|f\|_{L^p(\mathbb{R}^3)}~(|x|\ge 2R).
\end{gather}
Consequently, \eqref{R-R3-prop3-3} follows from \eqref{pr-R-R3-prop3-3} and \eqref{pr-R-R3-prop3-4}. This completes the proof of Proposition \ref{R-R3-prop3} (iii).
\end{proof}

\section{Zero-resolvent problem in the bounded domain} 
In this section, we study the corresponding zero-resolvent problem to \eqref{L-IBV-1} in the bounded domain $D$ with smooth boundary $\partial D$. Let $\mathbf{f}={}^\top(f_1,f_2,f_3)$ be a given pair of three functions $f_1=f_1(x)\in\mathbb{R},~f_2=f_2(x)\in\mathbb{R}^3$ and $f_3=f_3(x)\in\mathbb{R}^{3\times3}$. Then the zero-resolvent problem in $D$ is given as
\begin{equation}
\label{R-D-1}
A u=\mathbf{f}~\mathrm{in}~D,\quad
w|_{\partial D}=0.
\end{equation}
Here  $A$ is introduced in \eqref{def-L}. We then define the solution of \eqref{R-D-1} as $\mathcal{R}_{D}\mathbf{f}:={}^\top(\Phi_{D}\mathbf{f},\mathcal{W}_{D}\mathbf{f},\mathcal{G}_{D}\mathbf{f})$. The existence of \eqref{R-D-1} is given as follows.
\begin{prop}\label{R-D-prop2}
Let  $1<p<\infty$, let $D$ be a bounded domain in $\mathbb{R}^3$ with $C^3$-boundary, and let $\mathbf{f}\in X^p(D)$ with $\int_D f_1 dx=0$. Then \eqref{R-D-1} has a unique solution $u=\mathcal{R}_{D}\mathbf{f}\in Y^p(D)$ satisfying 
\begin{equation}
\left\|\mathcal{R}_D\mathbf{f}\right\|_{Y^p(D)}\leq 
C_{D}\|\mathbf{f}\|_{X^p(D)}.
\end{equation}
\end{prop}

\begin{rem}\label{R-D-rem1}
If we drop the assumption $\int_D f_1 dx=0$ in Proposition \ref{R-D-prop2}, the uniqueness of solutions to \eqref{R-D-1} does not  hold. In fact, if $u$ solves \eqref{R-D-1}, then $u+{}^\top (c,0,0)$ is also a solution to \eqref{R-D-1} in $D$ for any constant $c\in\mathbb{R}$.
\end{rem}

\begin{proof}[Proof of Proposition \ref{R-D-prop2}]
We only discuss the existence of the solution to \eqref{R-D-1}.  We rewrite \eqref{R-D-1} as
\begin{equation}
\label{R-D-2}
\begin{aligned}
\mathrm{div} w&=f_1~\mathrm{in}~D, \\
-\nu\Delta w-\tilde{\nu}\nabla\mathrm{div}w+\gamma^2\nabla\phi-\beta^2\mathrm{div}G&=f_2~\mathrm{in}~D, \\
-\nabla w&=f_3~\mathrm{in}~D, \\
w|_{\partial D}=0, \quad
\nabla\phi+\mathrm{div}{}^\top G&=0~\mathrm{in}~D.
\end{aligned}
\end{equation}
Taking the divergence to the third equation of \eqref{R-D-2}, we obtain $-\Delta w=\mathrm{div}f_3$ in $D$  with $w|_{\partial D}=0$. Thus, applying Lemma \ref{bdd-elliptic}, $w$ is uniquely solved as $w=\mathcal{W}_D\mathbf{f}=(-\Delta_{D})^{-1}\mathrm{div}f_3\in W_0^{1,p}(D)\cap W^{2,p}(D) $. Reading the first and third equation of \eqref{R-D-2} as $\nabla\mathrm{div} w=\nabla f_1$ and $\Delta w=-\mathrm{div} f_3$, and then substituting these to the second equation of \eqref{R-D-2}, we earn $\gamma^2\nabla \phi-\beta^2\mathrm{div}G=\tilde{\nu} \nabla f_1+f_2-\nu\mathrm{div} f_3.$
We next write $G=\nabla \psi$ with some $\psi\in (W_0^{1,p}(D)\cap W^{2,p}(D))^3$. Together with the last identity in \eqref{R-D-2}, the above equality leads to
\[
-\beta^2\Delta \psi-\gamma^2\nabla\mathrm{div} \psi=\tilde{\nu} \nabla f_1+f_2-\nu\mathrm{div} f_3~\mathrm{in}~D,\quad \psi|_{\partial D}=0.
\]
Therefore, we see from Lemma \ref{bdd-elliptic} that $\psi$ is uniquely determined as 
$
\psi=:((-\beta^2\Delta-\gamma^2\nabla\mathrm{div})_D)^{-1}(\tilde{\nu} \nabla f_1+f_2-\nu\mathrm{div} f_3),
$
which immediately leads to
$
G=\mathcal{G}_D \mathbf{f}=\nabla((-\beta^2\Delta-\gamma^2\nabla\mathrm{div})_D)^{-1}(\tilde{\nu}\nabla f_1+f_2-\nu\mathrm{div} f_3).
$ 
We finally focus on $\phi$. It follows from $\nabla \phi+\mathrm{div}{}^\top G=0$ and $\mathrm{div}{}^\top G=\nabla \mathrm{tr}G$ that $\phi+\mathrm{tr}G$ only takes some constant in $D$. Therefore, together with $\int_D \phi dx=0$, $\phi$ coincides $-\mathrm{tr} G$ in $D$. This yields
$
\phi=\Phi_D \mathbf{f}=-\mathrm{div}((-\beta^2\Delta-\gamma^2\nabla\mathrm{div})_D)^{-1}(\tilde{\nu} \nabla f_1+f_2-\nu\mathrm{div} f_3).
$ 
This completes the proof.
\end{proof}

\vskip5mm
\section{Resolvent problem in the exterior domain} 

In this section, we focus on the resolvent problem \eqref{R-Ex-1}. For the case $\lambda$ is away from 0, the existence and estimates of solutions to \eqref{R-Ex-1} are established by a similar method to Shibata and Tanaka \cite{ShTa04}. This fact is summarized by the following proposition.
\begin{prop}\label{R-Ex-Prop1}
Let $1<p<\infty$, $0<\varepsilon_0<\pi/2$, $\lambda_0>0$ and $\mathbf{f}\in X^p(\Omega)$. 
Then, if $\lambda\in \Xi_{\varepsilon_0}$, then \eqref{R-Ex-1} has a unique solution $u=(\phi,w,G)=(\lambda+A)^{-1}\mathbf{f}$
satisfying the following properties.

\vskip3mm
\noindent
(i) The operator $(\lambda+A)^{-1}$ belongs to $\mathcal{A}(\Xi_{\varepsilon_0}; \mathcal{L}(X^p(\Omega); Y^p(\Omega) ) )$.

\vskip3mm
\noindent
(ii) The solution $u$ satisfies
\begin{equation}\label{R-Ex-prop1-1}
|\lambda | \|u\|_{X^p(\Omega)}+\sum_{j=1}^2|\lambda|^{\frac{1-j}{2}} \|\nabla^j w\|_{L^p(\Omega)}\leq C_{\varepsilon_0 ,\lambda_0}\|\mathbf{f}\|_{X^p(\Omega)}
\end{equation}
for $\lambda\in \Xi_{\varepsilon_0}$ with $|\lambda|\ge \lambda_0$. Additionally, $(\lambda+A)^{-1}$ also belongs to $\mathcal{A}(\Xi_{\varepsilon_0}; \mathcal{L}(Y^p(\Omega) ) )$ and $u$ is controlled as
\begin{equation}\label{R-Ex-prop1-2}
|\lambda | \|u\|_{Y^p(\Omega)}\leq C_{\varepsilon_0 ,\lambda_0}\|\mathbf{f}\|_{Y^p(\Omega)}
\end{equation}
for $\mathbf{f}\in Y^p(\Omega)$ and $\lambda\in \Xi_{\varepsilon_0}$ with $|\lambda|\ge \lambda_0$.
\end{prop}

We next focus on the solution to \eqref{R-Ex-1} when $\lambda$ is near to $0$. This detail is displayed in the following proposition.

\begin{prop}\label{R-Ex-Prop2}
Let $1<p<\infty$, $0<\varepsilon_0<\pi/2$ and let $b, b_0>0$ be $\mathbb{R}^3 \setminus \Omega \subset B_{b_0}$ and $b>b_0$, and let $\mathbf{f}$ satisfy $\mathbf{f}\in X_b^p(\Omega)$. Let $D_{a_1,a_2}:=B_{a_2}\setminus B_{a_1}$ for $a_1,a_2\in \mathbb{R}$ with $a_2>a_1>0$.
Then, for any $b_1,b_2>0$ with $b_0<b_1<b_2<b$, there exist positive constant $\lambda_2=\lambda_2(\varepsilon_0,b_0,b_1,b_2,b,p)$ and operator $\widetilde{\mathcal{R}}_\Omega (\lambda)=\widetilde{\mathcal{R}}_{\varepsilon_0,b_0,b_1,b_2,b,p} (\lambda)$ such that the following assertions hold:
\begin{align}
{\rm (i)}~ &\widetilde{\mathcal{R}}_\Omega\in\mathcal{A}(\Dot{U}_{\lambda_2}; \mathcal{L}(X_b^p(\Omega),Y^p(\Omega_b))), \label{R-Ex-Prop2-1} \\
{\rm (ii)}~ &\widetilde{\mathcal{R}}_\Omega(\lambda)\mathbf{f}=(\lambda+A)^{-1}\mathbf{f}~(\lambda\in \Dot{U}_{\lambda_2}\cap \Xi_{\varepsilon_0},~\mathbf{f}\in X_b^p(\Omega) ), \label{R-Ex-Prop2-2} \\
{\rm (iii)}~
&\left\| \frac{d^m}{d\lambda^m}\widetilde{\mathcal{R}}_\Omega (\lambda)\mathbf{f}\right\|_{Y^p(\Omega_{b_1})} 
\le C\|\mathbf{f}\|_{X^p(\Omega)},
\label{R-Ex-Prop2-3} \\
&\left\| \frac{d^m}{d\lambda^m}\widetilde{\mathcal{R}}_\Omega (\lambda)\mathbf{f}\right\|_{Y^p( D_{b_1,b_2})} 
\le C\|\mathbf{f}\|_{X^p(\Omega)}
\begin{cases}
1 & (m=0), \\
\left|\log |\lambda|\right| & (m=1), \\
|\lambda|^{1-m} &(m\ge2),
\end{cases}  \label{R-Ex-Prop2-4} \\
&\left\| \frac{d^m}{d\lambda^m}\widetilde{\mathcal{R}}_\Omega (\lambda)\mathbf{f}\right\|_{Y^p(D_{b_2,b})} 
\le C\|\mathbf{f}\|_{X^p(\Omega)}
\begin{cases}
1 & (m\le1), \\
\left|\log |\lambda|\right| & (m=2), \\
|\lambda|^{2-m} &(m\ge3)
\end{cases}  \label{R-Ex-Prop2-5}
\end{align}
for $m\in\mathbb{N}_0$ and $\lambda\in \Dot{U}_{\lambda_2}$. 
Here $C=C_{\varepsilon_0,b_0,b_1,b_2,b,m,p}$ is a positive constant. 
\end{prop}

Since $b_1$ and $b_2$ in Proposition \ref{R-Ex-Prop2} are arbitrary, we can ignore \eqref{R-Ex-Prop2-4} and $(\lambda+A)^{-1}\mathbf{f}$ in $Y^p(\Omega_b)$ is estimated by the right-hand side of \eqref{R-Ex-Prop2-5}. This fact is summarized as the following corollary. 

\begin{cor}\label{R-Ex-Cor1} 
Assume that the same assumptions in Proposition \ref{R-Ex-Prop2} hold. Then there exists a positive constant $\lambda_3=\lambda_3(\varepsilon_0,b_0,b,p)$ such that $(\lambda+A)^{-1}\mathbf{f}$ satisfies
\begin{align}
&\left\| \frac{d^m}{d\lambda^m}(\lambda+A)^{-1}\mathbf{f}\right\|_{Y^p(\Omega_{b})} 
\le C\|\mathbf{f}\|_{X^p(\Omega)}
\begin{cases}
1 & (m\le 1), \\
\left|\log |\lambda|\right| &(m=2), \\
|\lambda|^{2-m} &(m\ge3)
\end{cases}
\label{R-Ex-Cor1-1} 
\end{align}
for $m\in\mathbb{N}_0$, $\lambda\in \Dot{U}_{\lambda_3}\cap \Xi_{\varepsilon_0}$ and $f\in X_b^p(\Omega)$. Here $C=C_{\varepsilon_0,b_0,b,m,p}$  is a positive constant. 
\end{cor}

\begin{proof}[Proof of Corollary \ref{R-Ex-Cor1} ]
We first estimate $(\lambda+A)^{-1}\mathbf{f}$ in $\Omega_{(b_0+b)/2}$. We set $b_1:=(b_0+b)/2$ and $b_2:=(b_0+b_1)/2=(3b_0+b)/4$ in Proposition \ref{R-Ex-Prop2}, and let $\lambda_{2,1}=\lambda_{2,1}(\varepsilon,b_0,b,m,p)$ be $\lambda_{2,1}:=\lambda_{2}(\varepsilon_0,b_0,(b_0+b)/2,(3b_0+b)/4, b,m,p)$. We then see from \eqref{R-Ex-Prop2-2} and \eqref{R-Ex-Prop2-3} that $(\lambda+A)^{-1}\mathbf{f}$ coincides $\widetilde{\mathcal{R}}_{\varepsilon_0,b_0,(b_0+b)/2,(3b_0+b)/4, b,m,p}(\lambda)\mathbf{f}$ for $\lambda\in \Dot{U}_{\lambda_{2,1}}\cap \Xi_{\varepsilon_0}$ and satisfies 
\begin{gather}\label{pr-R-Ex-Cor1-1}
\left\| \frac{d^m}{d\lambda^m}(\lambda+A)^{-1}\mathbf{f}\right\|_{Y^p\left(\Omega_{\frac{b_0+b}{2}}\right)} 
\le C\|\mathbf{f}\|_{X^p(\Omega)}
\end{gather}
for $m\in\mathbb{N}_0$, $\lambda\in \Dot{U}_{\lambda_{2,1}}\cap \Xi_{\varepsilon_0}$ and $f\in X_b^p(\Omega)$. Here $C$ depends only on $\varepsilon_0,b_0, b, m$ and $p$.
We next estimate $(\lambda+A)^{-1}\mathbf{f}$ in $D_{(b+R)/2,b}$. We set $b_1:=(b_0+b_2)/2=(b_0+3b)/4$ and $b_2:=(b_0+b)/2$ in Proposition \ref{R-Ex-Prop2}, and let $\lambda_{2,2}=\lambda_{2,2}(\varepsilon_0,b_0,b,p)$ be $\lambda_{2,2}:=\lambda_{2}(\varepsilon_0,b_0,(b_0+3b)/4,(b_0+b)/2, b)$. It then follows from \eqref{R-Ex-Prop2-2} and \eqref{R-Ex-Prop2-5} that $(\lambda+A)^{-1}\mathbf{f}$ is identified as $\widetilde{\mathcal{R}}_{\varepsilon_0,b_0,(b_0+3b)/4,(b_0+b)/2, b}(\lambda)\mathbf{f}$ for $\lambda\in \Dot{U}_{\lambda_{2,2} }\cap \Xi_{\varepsilon_0}$ and satisfies 
\begin{gather}\label{pr-R-Ex-Cor1-2}
\left\| \frac{d^m}{d\lambda^m}(\lambda+A)^{-1}\mathbf{f}\right\|_{Y^p(D_{\frac{b_0+b}{2},b})} 
\le C\|\mathbf{f}\|_{X^p(\Omega)}
\begin{cases}
1 & (m\le 1), \\
\left|\log |\lambda|\right| & (m=2), \\
|\lambda|^{2-m} &(m\ge3)
\end{cases}
\end{gather}
for $m\in\mathbb{N}_0$, $\lambda\in \Dot{U}_{\lambda_{2,2} }\cap \Xi_{\varepsilon_0}$ and $f\in X_b^p(\Omega)$. Here $C$ depends only on $\varepsilon_0,b_0, b, m$ and $p$.  Let $\lambda_3$ be $\lambda_3:=\min\{\lambda_{2,1},\lambda_{2,2}\}$ which depends only on $\varepsilon_0,b_0, b, m$ and $p$. Then taking $\lambda\in \Dot{U}_{\lambda_{3}}\cap \Xi_{\varepsilon_0}$ and $f\in X_b^p(\Omega)$ and combining \eqref{pr-R-Ex-Cor1-1} and \eqref{pr-R-Ex-Cor1-2}, we arrive at \eqref{R-Ex-Cor1-1}.
This completes the proof.
\end{proof}

In what follows, $p,~\varepsilon_0,~b_0,~b_1,~b_2,~b$ and $\mathbf{f}$ satisfy the assumptions of Proposition \ref{R-Ex-Prop2}. The proof of Proposition \ref{R-Ex-Prop2} is based on the parametrix method by using two operators $\mathcal{R}_{\mathbb{R}^3}(\lambda)$ and  $\mathcal{R}_{\Omega_{b}}$ appearing in Section 3 and Section 4 with $D=\Omega_{b}$, respectively. We first set $\mathbf{f}_0={}^\top(f_{1,0},f_{2,0},f_{3,0})$ as the zero extension of $\mathbf{f}$:
\[
\mathbf{f}_0(x):=
\begin{cases}
\mathbf{f}(x) & (x\in\Omega), \\
0 & (x\notin\Omega).
\end{cases}
\]
We next introduce a cut-off function $\varphi\in C_0^\infty(\mathbb{R}^3)$ satisfying $0\le \varphi(x) \le1~(x\in\mathbb{R}^3),~\varphi(x)=1~(|x|\le b_1)$
and $\varphi(x)=0~(|x|\ge b_2)$.
Using above notations, we define an operator $\mathcal{R}(\lambda)$ as
\begin{align*}
\mathcal{R}(\lambda)\mathbf{f}&={}^\top (\Phi(\lambda)\mathbf{f}, \mathcal{W}(\lambda)\mathbf{f},\mathcal{G}(\lambda)\mathbf{f}),  \quad
\mathcal{R}(\lambda)\mathbf{f}:=(1-\varphi)\mathcal{R}_{\mathbb{R}^3}(\lambda)\mathbf{f}_0+\varphi\mathcal{R}_{\Omega_{b}}(\mathbf{f}|_{\Omega_b}).
\end{align*}
Then $\mathcal{R}(\lambda)\mathbf{f}$ solves
\begin{equation}\label{pr-R-Ex-Prop2-1}
(\lambda+A)\mathcal{R}(\lambda)\mathbf{f}=\mathbf{f}+\mathcal{S}(\lambda)\mathbf{f},~ \mathcal{W}(\lambda)\mathbf{f}|_{\partial \Omega}=0,
\end{equation}
where $\mathcal{S}(\lambda)$ is an operator given as
\begin{align*}
&
\mathcal{S}(\lambda)\mathbf{f}={}^\top (\mathcal{S}_1(\lambda)\mathbf{f}, \mathcal{S}_2(\lambda)\mathbf{f},\mathcal{S}_3(\lambda)\mathbf{f}), \\
&
\mathcal{S}_1(\lambda)\mathbf{f}:=\lambda\varphi \Phi_{\Omega_{b}}(\mathbf{f}|_{\Omega_{b}})-(\mathcal{W}_{\mathbb{R}^3}(\lambda)\mathbf{f}_0-\mathcal{W}_{\Omega_{b}}(\mathbf{f}|_{\Omega_{b}}))\cdot\nabla\varphi, \\
&
\begin{aligned}
\mathcal{S}_2(\lambda)\mathbf{f}:=&
\lambda\varphi \mathcal{W}_{\Omega_{b}}(\mathbf{f}|_{\Omega_{b}})+2\nu\nabla\varphi: (\mathcal{W}_{\mathbb{R}^3}(\lambda)\mathbf{f}_0-\mathcal{W}_{\Omega_{b}}(\mathbf{f}|_{\Omega_{b}})) +\nu\Delta\varphi(\mathcal{W}_{\mathbb{R}^3}(\lambda)\mathbf{f}_0-\mathcal{W}_{\Omega_{b}}(\mathbf{f}|_{\Omega_{b}}))  \\
&\quad
+\tilde{\nu}\nabla((\mathcal{W}_{\mathbb{R}^3}(\lambda)\mathbf{f}_0-\mathcal{W}_{\Omega_{b}}(\mathbf{f}|_{\Omega_{b}}))\cdot \nabla\varphi)
+\tilde{\nu}\mathrm{div}(\mathcal{W}_{\mathbb{R}^3}(\lambda)\mathbf{f}_0-\mathcal{W}_{\Omega_{b}}(\mathbf{f}|_{\Omega_{b}}))\nabla\varphi \\
&\qquad -\gamma^2 (\Phi_{\mathbb{R}^3}(\lambda)\mathbf{f}_0-\Phi_{\Omega_{b}}(\mathbf{f}|_{\Omega_{b}}))\nabla\varphi  +\beta^2(\mathcal{G}_{\mathbb{R}^3}(\lambda)\mathbf{f}_0-\mathcal{G}_{\Omega_{b}}(\mathbf{f}|_{\Omega_{b}}))\nabla\varphi,
\end{aligned}
\\
&
\mathcal{S}_3(\lambda)\mathbf{f}:=\lambda\varphi \mathcal{G}_{\Omega_{b}}(\mathbf{f}|_{\Omega_{b}})+(\mathcal{W}_{\mathbb{R}^3}(\lambda)\mathbf{f}_0-\mathcal{W}_{\Omega_{b}}(\mathbf{f}|_{\Omega_{b}}))\otimes  \nabla\varphi.
\end{align*}
It follows from Proposition \ref{R-R3-prop2}, \ref{R-R3-prop3} and \ref{R-D-prop2} that $\mathcal{S}(\lambda)$ satisfies
\begin{align*}
&\mathcal{S}(\lambda)\in \mathcal{A}(\dot{U}_{\lambda_1}, \mathcal{L}(W_b^{1,p}(\Omega)\times L_b^p(\Omega)^3\times W_b^{1,p}(\Omega)^{3\times3})), \\
&\|(\mathcal{S}(\lambda)-\mathcal{S}(0))\mathbf{f}\|_{X^p(\Omega_b)} \le C|\lambda|\|\mathbf{f}\|_{X^p(\Omega)}~(\lambda\in \dot{U}_{\lambda_1},~\mathbf{f}\in X_b^p(\Omega)).
\end{align*}  
However, it is not guaranteed that $\mathcal{S}(\lambda)$ belongs to $\mathcal{A}(\dot{U}_{\lambda_1}, \mathcal{L}(X_b^p(\Omega)) )$. Therefore, we have to modify $\mathcal{S}(\lambda)$ with some operator $\mathcal{S}^3(\lambda)$ satisfying
\begin{align}
&\mathcal{S}^3(\lambda)\in \mathcal{A}(\dot{U}_{\lambda_1}, \mathcal{L}(X_b^p(\Omega)) ), \quad \|(\mathcal{S}^3(\lambda)-\mathcal{S}^3(0))\mathbf{f}\|_{X^p(\Omega_b)} \le C|\lambda|\|\mathbf{f}\|_{X^p(\Omega)}~(\lambda\in \dot{U}_{\lambda_1},~\mathbf{f}\in X_b^p(\Omega)). \label{pr-R-Ex-Prop2-1-1}
\end{align} 
We first prepare a cut-off function $\varphi_1\in C_0^\infty(\mathbb{R}^3)$ such that
$
0\le \varphi_1\le 1,~\int_\Omega \varphi_1 dx=1,~\textrm{supp}\varphi_1\subset D_{b_1,b_2}:=B_{b_2}\setminus B_{b_1},
$
and set $\mathcal{S}^{1}(\lambda)$ and $\mathcal{R}^1(\lambda)$ as
\begin{align*}
\mathcal{S}^{1}(\lambda)\mathbf{f}&={}^\top(\mathcal{S}_1^{1}(\lambda)\mathbf{f},\mathcal{S}_2^{1}(\lambda)\mathbf{f},\mathcal{S}_3^{1}(\lambda)\mathbf{f}); \\
\mathcal{S}_1^{1}(\lambda)\mathbf{f}&:= \mathcal{S}_1(\lambda)\mathbf{f}+\lambda\left(\int_{B_{b}} \varphi\Phi_{\mathbb{R}^3}(\lambda)\mathbf{f}_0dx-\int_{\Omega_{b}}\varphi\Phi_{\Omega_{b_3}}(\mathbf{f}|_{\Omega_{b}})dx\right)\varphi_1,
\\
\mathcal{S}_2^{1}(\lambda)\mathbf{f}&:=\mathcal{S}_2(\lambda)\mathbf{f}+\gamma^2\left(\int_{B_{b}} \varphi\Phi_{\mathbb{R}^3}(\lambda)\mathbf{f}_0dx-\int_{\Omega_{b}}\varphi\Phi_{\Omega_{b}}(\mathbf{f}|_{\Omega_{b}})dx\right)\nabla\varphi_1,\\
\mathcal{S}_3^{1}(\lambda)\mathbf{f}&:=\mathcal{S}_3(\lambda)\mathbf{f},\\
\mathcal{R}^{1}(\lambda)\mathbf{f}&={}^\top(\Phi^1(\lambda)\mathbf{f},\mathcal{W}^{1}(\lambda)\mathbf{f},\mathcal{G}^{1}(\lambda)\mathbf{f}); \\
\Phi^1(\lambda)\mathbf{f}&:=\Phi(\lambda)\mathbf{f}+\left(\int_{B_{b}} \varphi\Phi_{\mathbb{R}^3}(\lambda)\mathbf{f}_0dx-\int_{\Omega_{b}}\varphi\Phi_{\Omega_{b}}(\mathbf{f}|_{\Omega_{b}})dx\right)\varphi_1, \\
\mathcal{W}^1(\lambda)\mathbf{f}&:=\mathcal{W}(\lambda)\mathbf{f}, ~
\mathcal{G}^1(\lambda)\mathbf{f}:=\mathcal{G}(\lambda)\mathbf{f}.
\end{align*}
Then, $\mathcal{S}^{1}(\lambda)$ and $\mathcal{R}^1(\lambda)$ satisfy
\begin{gather*}
(\lambda+A)\mathcal{R}^1(\lambda)\mathbf{f}=\mathbf{f}+\mathcal{S}^{1}(\lambda)\mathbf{f}~\textrm{in}~\Omega,~\mathcal{W}^1(\lambda)\mathbf{f}|_{\partial \Omega}=0, \\
\int_{\Omega_b} \Phi^{1}(\lambda)\mathbf{f}dx=
\int_{\Omega_b} \mathcal{S}_1^{1}(\lambda)\mathbf{f}dx=0,\quad \mathcal{S}^{1}(\lambda)\mathbf{f}(x)=0~(|x|\ge b).
\end{gather*}
We next set $\mathcal{S}^{2}(\lambda)$ and $\mathcal{R}^2(\lambda)$ as
\begin{align*}
\mathcal{S}^{2}(\lambda)\mathbf{f}&={}^\top(\mathcal{S}_1^{2}(\lambda)\mathbf{f},\mathcal{S}_2^{2}(\lambda)\mathbf{f},\mathcal{S}_3^{2}(\lambda)\mathbf{f}); \\
\mathcal{S}_1^{2}(\lambda)\mathbf{f}&:=\mathcal{S}_1^{1}(\lambda)\mathbf{f},\\
\mathcal{S}_2^{2}(\lambda)\mathbf{f}&:=\mathcal{S}_2^{1}(\lambda)\mathbf{f}+\beta^2\mathrm{div}\left[ \left\{(-\Delta_{\mathbb{R}^3})^{-1}\mathrm{div}\mathcal{G}_{\mathbb{R}^3}(\lambda)\mathbf{f}_0-(-\Delta_{\Omega_b})^{-1}\mathrm{div}\mathcal{G}_{\Omega_{b}}(\mathbf{f}|_{\Omega_{b}})\right\} \otimes\nabla\varphi\right],\\
\mathcal{S}_3^{2}(\lambda)\mathbf{f}&:=\mathcal{S}_3^{1}(\lambda)\mathbf{f}+\lambda\left\{(-\Delta_{\mathbb{R}^3})^{-1}\mathrm{div}\mathcal{G}_{\mathbb{R}^3}(\lambda)\mathbf{f}_0-(-\Delta_{\Omega_b})^{-1}\mathrm{div}\mathcal{G}_{\Omega_{b}}(\mathbf{f}|_{\Omega_{b}})\right\}\otimes\nabla\varphi,\\
\mathcal{R}^{2}(\lambda)\mathbf{f}&={}^\top(\Phi^2(\lambda)\mathbf{f},\mathcal{W}^{2}(\lambda)\mathbf{f},\mathcal{G}^{2}(\lambda)\mathbf{f}); \\
\Phi^2(\lambda)\mathbf{f}&:=\Phi^1(\lambda)\mathbf{f}, \quad
\mathcal{W}^2(\lambda)\mathbf{f}:=\mathcal{W}^{1}(\lambda)\mathbf{f}, \\
\mathcal{G}^2(\lambda)\mathbf{f}&:=\mathcal{G}^{1}(\lambda)\mathbf{f}+ \left\{(-\Delta_{\mathbb{R}^3})^{-1}\mathrm{div}\mathcal{G}_{\mathbb{R}^3}(\lambda)\mathbf{f}_0-(-\Delta_{\Omega_b})^{-1}\mathrm{div}\mathcal{G}_{\Omega_{b}}(\mathbf{f}|_{\Omega_{b}})\right\}\otimes\nabla\varphi.
\end{align*}
Then, $\mathcal{S}^{2}(\lambda)$ and $\mathcal{R}^2(\lambda)$ satisfy
\begin{gather*}
(\lambda+A)\mathcal{R}^2(\lambda)\mathbf{f}=\mathbf{f}+\mathcal{S}^{2}(\lambda)\mathbf{f}~\textrm{in}~\Omega,~\mathcal{W}^2(\lambda)\mathbf{f}|_{\partial \Omega}=0, \\
\int_{\Omega} \Phi^{2}(\lambda)\mathbf{f}dx=\int_{\Omega} \mathcal{S}_1^{2}(\lambda)\mathbf{f}dx=0,\quad \mathcal{G}^{2}(\lambda)\mathbf{f},~
 \mathcal{S}_1^{2}(\lambda)\mathbf{f}\in \nabla (W_0^{1,p}(\Omega))^3, \\
 \mathcal{S}^{2}(\lambda)\mathbf{f}(x)=0~(|x|\ge b).
\end{gather*}
We finally define $\mathcal{S}^{3}(\lambda)$ and $\mathcal{R}^3(\lambda)$ as
\begin{align*}
\mathcal{S}^{3}(\lambda)\mathbf{f}&={}^\top(\mathcal{S}_1^{3}(\lambda)\mathbf{f},\mathcal{S}_2^{3}(\lambda)\mathbf{f},\mathcal{S}_3^{3}(\lambda)\mathbf{f}); \\
\mathcal{S}_1^{3}(\lambda)\mathbf{f}&:=\mathcal{S}_1^2(\lambda)\mathbf{f}-\lambda\left\{(-\Delta_{\mathbb{R}^3})^{-1}\mathrm{div}\mathcal{G}_{\mathbb{R}^3}(\lambda)\mathbf{f}_0-(-\Delta_{\Omega_{b}})^{-1}\mathrm{div}\mathcal{G}_{\Omega_{b}}(\mathbf{f}|_{\Omega_{b}})\right\}\cdot\nabla\varphi \\
&\quad +\lambda\left[\int_{\Omega_{b}}\left\{(-\Delta_{\mathbb{R}^3})^{-1}\mathrm{div}\mathcal{G}_{\mathbb{R}^3}(\lambda)\mathbf{f}_0-(-\Delta_{\Omega_{b}})^{-1}\mathrm{div}\mathcal{G}_{\Omega_{b}}(\mathbf{f}|_{\Omega_{b}})\right\}\cdot\nabla\varphi dx\right] \varphi_1,\\
\mathcal{S}_2^{3}(\lambda)\mathbf{f}&:=\mathcal{S}_2^2(\lambda)\mathbf{f}-\gamma^2\nabla\left[\left\{(-\Delta_{\mathbb{R}^3})^{-1}\mathrm{div}\mathcal{G}_{\mathbb{R}^3}(\lambda)\mathbf{f}_0-(-\Delta_{\Omega_{b}})^{-1}\mathrm{div}\mathcal{G}_{\Omega_{b}}(\mathbf{f}|_{\Omega_{b}})\right\}\cdot\nabla\varphi\right] \\
&\quad +\gamma^2\left[\int_{\Omega_{b}}\left\{(-\Delta_{\mathbb{R}^3})^{-1}\mathrm{div}\mathcal{G}_{\mathbb{R}^3}(\lambda)\mathbf{f}_0-(-\Delta_{\Omega_{b}})^{-1}\mathrm{div}\mathcal{G}_{\Omega_{b}}(\mathbf{f}|_{\Omega_{b}})\right\}\cdot\nabla\varphi dx\right] \nabla\varphi_1,\\[1ex]
\mathcal{S}_3^{3}(\lambda)\mathbf{f}&:=\mathcal{S}_3^2(\lambda)\mathbf{f},\\
\mathcal{R}^{3}(\lambda)\mathbf{f}&={}^\top(\Phi^3(\lambda)\mathbf{f},\mathcal{W}^{3}(\lambda)\mathbf{f},\mathcal{G}^{3}(\lambda)\mathbf{f}); \\
\Phi^3(\lambda)\mathbf{f}&:=\Phi^2(\lambda)\mathbf{f}-\left[(-\Delta_{\mathbb{R}^3})^{-1}\mathrm{div}\mathcal{G}_{\mathbb{R}^3}(\lambda)\mathbf{f}_0-(-\Delta_{\Omega_{b}})^{-1}\mathrm{div}\mathcal{G}_{\Omega_{b}}(\mathbf{f}|_{\Omega_{b}})\right]\cdot\nabla\varphi \\
&\quad +\left[\int_{\Omega_{b}} \left\{(-\Delta_{\mathbb{R}^3})^{-1}\mathrm{div}\mathcal{G}_{\mathbb{R}^3}(\lambda)\mathbf{f}_0-(-\Delta_{\Omega_{b}})^{-1}\mathrm{div}\mathcal{G}_{\Omega_{b_3}}(\mathbf{f}|_{\Omega_{b}})\right\}\cdot\nabla\varphi dx\right] \varphi_1, \\
\mathcal{W}^3(\lambda)\mathbf{f}&:=\mathcal{W}^2(\lambda)\mathbf{f}, \quad
\mathcal{G}^3(\lambda)\mathbf{f}:=\mathcal{G}^2(\lambda)\mathbf{f}.
\end{align*}
Then, $\mathcal{S}^{3}(\lambda)$ and $\mathcal{R}^3(\lambda)$ satisfy
\begin{gather}
(\lambda+A)\mathcal{R}^3(\lambda)\mathbf{f}=\mathbf{f}+\mathcal{S}^{3}(\lambda)\mathbf{f}~\textrm{in}~\Omega,~\mathcal{W}^3(\lambda)\mathbf{f}|_{\partial \Omega}=0, \label{pr-R-Ex-Prop2-3} \\
\int_{\Omega} \Phi^{3}(\lambda)\mathbf{f}dx=\int_{\Omega} \mathcal{S}_1^{3}(\lambda)\mathbf{f}dx=0, \quad \mathcal{G}^{3}(\lambda)\mathbf{f},~\mathcal{S}_1^{3}(\lambda)\mathbf{f} \in\nabla (W_0^{1,p}(\Omega))^3, \notag
\\
\nabla \Phi^3(\lambda)\mathbf{f}+\mathrm{div} {}^\top \mathcal{G}^3(\lambda)\mathbf{f}=\nabla \mathcal{S}_1^3(\lambda)\mathbf{f}+\mathrm{div} {}^\top \mathcal{S}_3^3(\lambda)\mathbf{f}=0~\mathrm{in}~\Omega,~\mathcal{S}^3(\lambda)\mathbf{f}=0~(|x|\ge b). \notag
\end{gather}
Here we use the facts $\nabla (\Phi_{\mathbb{R}^3}(\lambda)\mathbf{f}_0) +\mathrm{div}{}^\top (\mathcal{G}_{\mathbb{R}^3}(\lambda)\mathbf{f}_0)=0, \nabla (\Phi_{\Omega_b}(\mathbf{f}|_{\Omega_b})) +\mathrm{div}{}^\top (\mathcal{G}_{\Omega_b}(\mathbf{f}|_{\Omega_b}))=0,~\mathcal{G}_{\mathbb{R}^3}(\lambda)\mathbf{f}_0\in \nabla (W^{1,p}(\mathbb{R}^3))^3$ and $\mathcal{G}_{\Omega_b}(\mathbf{f}|_{\Omega_b})\in \nabla (W_0^{1,p}(\Omega_b))^3$.
Moreover, we see from Propositions \ref{R-R3-prop2}, \ref{R-R3-prop3} and \ref{R-D-prop2} that \eqref{pr-R-Ex-Prop2-1-1} hold.
This gives the desired property.

We next investigate the invertibility of $I+\mathcal{S}^3(\lambda)$ with small $\lambda$. This fact is enough to focus on the case $\lambda=0$. 
\begin{lem}\label{R-Ex-lem1}
For $1<p<\infty$, $I+\mathcal{S}^3(0)$ has its bounded inverse $(I+\mathcal{S}^3(0))^{-1}\in \mathcal{L}(X_b^p(\Omega))$.
\end{lem}
The following uniqueness of \eqref{R-Ex-1} with $\lambda=0$ in the locally Sobolev space is crucial to earn the invertibility of $I+\mathcal{S}^3(0)$.
\begin{lem}\label{R-Ex-lem2}
Let $1<p<\infty$ and let $\Omega$ be an exterior domain in $\mathbb{R}^3$ with $C^3$ boundary. Let $u={}^\top(\phi,w,G)$ be a solution to the following problem:
\begin{gather}
Au=0~\mathrm{in}~\Omega, w|_{\partial \Omega}=0, \label{R-Ex-lem2-1} \\
 u(x)=O(|x|^{-2})~\mathrm{as}~|x|\to\infty, \label{R-Ex-lem2-2} \\
\nabla\phi+\mathrm{div} {}^\top G=0 ~\mathrm{in}~\Omega. \label{R-Ex-lem2-3}
\end{gather}
If $u$ belongs to $W_{\mathrm{loc}}^{1,p}(\Omega)\times W_{\mathrm{loc}}^{2,p}(\Omega)^3 \times W_{\mathrm{loc}}^{1,p}(\Omega)^{3\times3}$ and $G$ is written as $G=\nabla\psi$ with some $\psi\in W_{\mathrm{loc}}^{2,p}(\Omega)^3$ satisfying $\psi|_{\partial \Omega}=0$ and $\psi(x)=O(|x|^{-1})$ as $|x|\to\infty$, then $u=0$ in $\Omega$ holds.
\end{lem}
\begin{proof}
It is easy to see from \eqref{R-Ex-lem2-1} that $w=0$ in $\Omega$ holds. Combining this fact and \eqref{R-Ex-lem2-1} gives $\gamma^2\nabla\phi-\beta^2\mathrm{div} G=0$ in $\Omega$. Substituting \eqref{R-Ex-lem2-3} and $G=\nabla\psi$ into the resultant equation, we arrive at
\begin{gather}
-\beta^2\Delta\psi-\gamma^2\nabla\mathrm{div}\psi=0~\mathrm{in}~\Omega, \label{pr-R-Ex-lem2-1} \\
\psi|_{\partial \Omega}=0,~\nabla^k\psi(x)=O(|x|^{-1-k})~\mathrm{as}~|x|\to\infty,~k=0,1. \label{pr-R-Ex-lem2-2}
\end{gather}
We also note that since \eqref{pr-R-Ex-lem2-1} is an elliptic equation, $\psi$ also belongs to $W_{\mathrm{loc}}^{2,2}(\Omega)$.  We are now going to show $\psi=0$ in $\Omega$. Let $\theta\in C_0^\infty(\mathbb{R}^3)$ be a cut-off function satisfying $\theta(x)=1~(|x|<1)$ and  $\theta(x)=0~(|x|>2)$. We then define $\theta_R\in C_0^\infty(\mathbb{R}^3)$ as $\theta_R(x):=\theta(x/R)$ for $R>0$. 
Let $(\cdot,\cdot)_{\Omega_{2R}}$ be the inner product of $L^2(\Omega_{2R})$ denoted by
$
(f,g)_D:=\int_{\Omega_{2R} } f(x)g(x)dx~(f,g\in L^2(\Omega_{2R})).
$
Taking the $L^2(\Omega_{2R} )$-inner product of \eqref{pr-R-Ex-lem2-1} with $\theta_R\psi$ and applying integration by parts, we have
\begin{equation}
\beta^2(\nabla\psi, \theta_R\nabla\psi)_{\Omega_{2R}}+\gamma^2(\mathrm{div}\psi, \theta_R\mathrm{div}\psi)_{\Omega_{2R}} \\
=-\beta^2(\nabla\psi, \psi\otimes\nabla\theta_R)_{\Omega_{2R}}-\gamma^2(\mathrm{div}\psi,\psi\cdot\nabla \theta_R)_{\Omega_{2R}}.
\label{pr-R-Ex-lem2-3}
\end{equation}
It follows from $|\nabla \theta_R(x)|\le CR^{-1}$ and \eqref{pr-R-Ex-lem2-2} that the right-hand side of \eqref{pr-R-Ex-lem2-3} is controlled by $CR^{-1}$, provided that $R$ is sufficiently large. Therefore taking the limit $R\to\infty$ in \eqref{pr-R-Ex-lem2-3}, we earn $\nabla\psi=0$ in $\Omega$. Therefore, thanks to the second condition of \eqref{pr-R-Ex-lem2-2} with $k=0$, $\psi$ should be $0$ in $\Omega$. As a result, $G=0$ in $\Omega$ is guranteed. From \eqref{R-Ex-lem2-2} and \eqref{R-Ex-lem2-3}, we have $\phi=0$ in $\Omega$. This completes the proof. 
\end{proof}

\begin{proof}[Proof of Lemma \ref{R-Ex-lem1}]
Since $I+\mathcal{S}^3(0)$ is a compact operator on $X_b^p(\Omega)$ onto itself, it suffices to show its injectivity by virtue of Fredholm's alternative theorem. Let $\mathbf{f}\in X_b^p(\Omega)$ satisfy $(I+\mathcal{S}^3(0))\mathbf{f}=0$. We then see that $\mathbf{f}=0$ holds in $\mathbb{R}^3\setminus D_{b_1,b_2}$, where $D_{b_1,b_2}:=B_{b_2}\setminus B_{b_1}$.
Setting $\lambda=0$ in \eqref{pr-R-Ex-Prop2-3} gives
\begin{gather*}
A\mathcal{R}^3(0)\mathbf{f}=0~\mathrm{in}~\Omega, \quad \mathcal{W}^3(0)\mathbf{f}|_{\partial \Omega}=0, \\
 \mathcal{G}^{3}(0)\mathbf{f}=\nabla\Psi^3(0)\mathbf{f},\quad
\nabla \Phi^3(0)\mathbf{f}+\mathrm{div} {}^\top \mathcal{G}^3(0)\mathbf{f}=0~\mathrm{in}~\Omega, \\
\int_{\Omega_{b}} \Phi^{3}(0)\mathbf{f}dx=0. 
\end{gather*}
Here $\Psi^3(0)\mathbf{f}$ denotes 
$
\Psi^3(0)\mathbf{f}:=(1-\varphi)(-\Delta_{\mathbb{R}^3} )^{-1}\mathrm{div}\mathcal{G}_{\mathbb{R}^3}(0)\mathbf{f}_0+\varphi(-\Delta_{\Omega_{b}})^{-1}\mathrm{div}\mathcal{G}_{\Omega_{b}}(\mathbf{f}|_{\Omega_{b}})
$
and satisfies $\Psi^3(0)\mathbf{f}\in W_{\mathrm{loc}}^{2,p}(\Omega)^3$ and $\Psi^3(0)\mathbf{f}|_{\partial \Omega}=0$. Moreover, it follows from Proposition \ref{R-R3-prop3} that $\mathcal{R}^3(0)\mathbf{f}=O(|x|^{-2})$ and $\Psi^3(0)\mathbf{f}=O(|x|^{-1})$ hold as $|x|\to\infty$. Hence we are able to apply Lemma \ref{R-Ex-lem2} to obtain $\mathcal{R}^3(0)\mathbf{f}=0$ in $\Omega$. This yields
\begin{align}
\mathcal{R}_{\mathbb{R}^3}(0)\mathbf{f}_0=0,~(-\Delta_{\mathbb{R}^3} )^{-1}\mathrm{div}\mathcal{G}_{\mathbb{R}^3}(0)\mathbf{f}_0=0~(|x|\ge b_2),  \label{pr-R-Ex-lem1-1} \\ 
\mathcal{R}_{\Omega_{b}}(\mathbf{f}|_{\Omega_{b_3}})=0,~(-\Delta_{\Omega_{b_3}})^{-1}\mathrm{div}\mathcal{G}_{\Omega_{b}}(\mathbf{f}|_{\Omega_{b}})=0~(x\in \Omega_{b_1}). \label{pr-R-Ex-lem1-2}
\end{align}
We next extend $\mathcal{R}_{\Omega_{b}}(\mathbf{f}|_{\Omega_{b}})={}^\top(\Phi_{\Omega_{b}}(\mathbf{f}|_{\Omega_{b}}),\mathcal{W}_{\Omega_{b}}(\mathbf{f}|_{\Omega_{b}}),\mathcal{G}_{\Omega_{b}}(\mathbf{f}|_{\Omega_{b}}))$ and $(-\Delta_{\Omega_{b}})^{-1}$ $\mathrm{div}\mathcal{G}_{\Omega_{b}}(\mathbf{f}|_{\Omega_{b}})$ as
\begin{gather*}
\widetilde{u}={}^\top (\widetilde{\phi},\widetilde{w},\widetilde{G})
:=
\begin{cases}
\mathcal{R}_{\Omega_{b}}(\mathbf{f}|_{\Omega_{b}}) & (x\in\Omega_{b}), \\
0 & (x\notin\Omega_{b}),
\end{cases} 
\quad
\tilde{\psi}:=
\begin{cases}
(-\Delta_{\Omega_{b}})^{-1}\mathrm{div}\mathcal{G}_{\Omega_{b}}(\mathbf{f}|_{\Omega_{b}}) & (x\in\Omega_{b}), \\
0 & (x\notin\Omega_{b}),
\end{cases}
\end{gather*}
respectively.
In view of \eqref{pr-R-Ex-lem1-2}, $\widetilde{u}$ belongs to $W^{1,p}(B_{b})\times W^{2,p}(B_{b})^3\times W^{1,p}(B_{b})^{3\times3}$ and solves
\begin{gather}
A\widetilde{u}=\mathbf{f}_0~\mathrm{in}~B_{b}, \quad
\widetilde{w}|_{\partial B_{b}}=0, \label{pr-R-Ex-lem1-3}\\
\widetilde{G}\in \nabla(W^{2,p}(B_{b})\cap W_0^{1,p}(B_{b}))^3,~\nabla\tilde{\phi}+\mathrm{div}{}^\top \tilde{G}=0~\mathrm{in}~B_b, \label{pr-R-Ex-lem1-4}\\
\int_{B_{b}}\widetilde{\phi}dx=0. \label{pr-R-Ex-lem1-5}
\end{gather}
Furthermore, in view of $b>b_2$ and \eqref{pr-R-Ex-lem1-1}, $\mathcal{R}_{\mathbb{R}^3}(0)\mathbf{f}_0$ also satisfies \eqref{pr-R-Ex-lem1-3} and \eqref{pr-R-Ex-lem1-4}. Therefore, it follows from Proposition \ref{R-D-prop2} and Remark \ref{R-D-rem1} that there exists a constant $c_0\in\mathbb{R}$ such that
\begin{equation}
\mathcal{R}_{\mathbb{R}^3}(0)\mathbf{f}_0-{}^\top(c_0,0,0)=\widetilde{u},~(-\Delta_{\mathbb{R}^3} )^{-1}\mathrm{div}\mathcal{G}_{\mathbb{R}^3}(0)\mathbf{f}_0=\widetilde{\psi}~\mathrm{in}~B_{b}. \label{pr-R-Ex-lem1-6}
\end{equation}
We next prove $c_0=0$. Integrating $\Phi^3(0)\mathbf{f}=0$ over $\Omega_{b}$ and noticing $1-\varphi=0~(x\notin \Omega)$ yields
$\int_{B_{b}} \Phi_{\mathbb{R}^3}(0)\mathbf{f} dx=\int_{\Omega_{b}} \Phi^3(0)\mathbf{f} dx =0 $.
This computation and \eqref{pr-R-Ex-lem1-5} lead to $c_0=0$. As a result, \eqref{pr-R-Ex-lem1-6} leads to $\mathcal{R}_{\mathbb{R}^3}(0)\mathbf{f}_0=\mathcal{R}_{\Omega_{b}}(\mathbf{f}|_{\Omega_{b}})$ and $(-\Delta_{\mathbb{R}^3} )^{-1}\mathrm{div}\mathcal{G}_{\mathbb{R}^3}(0)\mathbf{f}_0=(-\Delta_{\Omega_{b}})^{-1}\mathrm{div}\mathcal{G}_{\Omega_{b}}(\mathbf{f}|_{\Omega_{b}})$ in $\Omega_{b}$. This gives $S^{3}(0)\mathbf{f}=0$ in $\Omega_{b}$, which implies $\mathbf{f}=0$ in $\Omega$. This completes the proof.
\end{proof}

\begin{proof}[Proof of Proposition \ref{R-Ex-Prop2}]
We once assume the invertibility of $I+\mathcal{S}^3(\lambda)$ with small $\lambda$. Since $I+\mathcal{S}^3(\lambda)$ is rewritten as
$I+\mathcal{S}^3(\lambda)=(I+\mathcal{S}^3(0))\left[I+(I+\mathcal{S}^3(0))^{-1}(\mathcal{S}^3(\lambda)-\mathcal{S}^3(0))\right],$ we have the Neumann series:
\begin{equation}\label{pr-R-Ex-Prop2-4}
(I+\mathcal{S}^3(\lambda))^{-1}=\sum_{N=0}^\infty \left[(I+\mathcal{S}^3(0))^{-1} (\mathcal{S}^3(\lambda)-\mathcal{S}^3(0))\right]^N(I+\mathcal{S}^3(0))^{-1}
\end{equation}
in $\mathcal{L}(X_b^p(\Omega))$.
These facts are justified as the follows. The invertibility of $I+\mathcal{S}^3(0)$ is guaranteed by Lemma \ref{R-Ex-lem1}, and therefore we focus on the convergence of right-hand side of \eqref{pr-R-Ex-Prop2-4}. Noticing \eqref{pr-R-Ex-Prop2-1-1}, we can take $\lambda_2\in (0,\lambda_1]$ small such that
$
\left\| (I+\mathcal{S}^3(0))^{-1} (\mathcal{S}^3(\lambda)-\mathcal{S}^3(0))\right\|_{\mathcal{L}(X_b^p(\Omega))} \le 1/2$ for $\lambda\in \dot{U}_{\lambda_2}.
$
This implies that the right-hand side of \eqref{pr-R-Ex-Prop2-4} uniformly converges in $\mathcal{L}(X_b^p(\Omega))$. Consequently, $I+\mathcal{S}^3(\lambda)$ has its inverse for $\lambda\in \dot{U}_{\lambda_2}$ and $(I+\mathcal{S}^3(\lambda))^{-1}$ belongs to $ \mathcal{A}(\Dot{U}_{\lambda_2},\mathcal{L}(X_b^p(\Omega))).$
We then define $\widetilde{\mathcal{R}}_{\Omega}(\lambda)$ as
$
\widetilde{\mathcal{R}}_{\Omega}(\lambda):=\mathcal{R}^3(\lambda)(I+\mathcal{S}^3(\lambda))^{-1}.
$
Taking $(I+\mathcal{S}^3(\lambda))^{-1}$ in \eqref{pr-R-Ex-Prop2-3}, $\widetilde{\mathcal{R}}_{\Omega}(\lambda)\mathbf{f}$ solves \eqref{R-Ex-1} for $\lambda\in \dot{U}_{\lambda_2}$. Therefore, restricting $\lambda\in \dot{U}_{\lambda_2}\cap \Xi_{\varepsilon_0}$ and applying Proposition \ref{R-Ex-Prop1}, $\widetilde{\mathcal{R}}_{\Omega}(\lambda)\mathbf{f}=(\lambda+A)^{-1}\mathbf{f}$ is guranteed. This proves \eqref{R-Ex-Prop2-1} and \eqref{R-Ex-Prop2-2}. To show \eqref{R-Ex-Prop2-3} and \eqref{R-Ex-Prop2-4}, we first estimate the derivative of $\mathcal{R}^3(\lambda)\mathbf{f}$ with respect to $\lambda$. Let $m\ge 0$. Then the each components of $(d^m/d\lambda^m)\mathcal{R}^3(\lambda)\mathbf{f}$ are computed as
\begin{align*}
\frac{d^m}{d\lambda^m}\Phi^3(\lambda)\mathbf{f}&=(1-\varphi)\frac{d^m}{d\lambda^m}\Phi_{\mathbb{R}^3}(\lambda)\mathbf{f}_0+\frac{d^m}{d\lambda^m}(-\Delta_{\mathbb{R}^3})^{-1}\mathrm{div}\mathcal{G}_{\mathbb{R}^3}(\lambda)\mathbf{f}_0\cdot \nabla\varphi\\
&\quad +\delta_{m,0}\left\{\varphi\Phi_{\Omega_{b}}(\mathbf{f}|_{\Omega_{b}})+(-\Delta_{\Omega_b})^{-1}\mathrm{div}\mathcal{G}_{\Omega_b}(\mathbf{f}|_{\Omega_{b}})\cdot\nabla\varphi\right\}, \\
\frac{d^m}{d\lambda^m}\mathcal{W}^3(\lambda)\mathbf{f}&=(1-\varphi)\frac{d^m}{d\lambda^m} \mathcal{W}_{\mathbb{R}^3}(\lambda)\mathbf{f}_0+\delta_{m,0}\varphi\mathcal{W}_{\Omega_{b}}(\mathbf{f}|_{\Omega_{b}}), \\
\frac{d^m}{d\lambda^m}\mathcal{G}^3(\lambda)\mathbf{f}&=(1-\varphi)\frac{d^m}{d\lambda^m} \mathcal{G}_{\mathbb{R}^3}(\lambda)\mathbf{f}_0+ \frac{d^m}{d\lambda^m}(-\Delta_{\mathbb{R}^3})^{-1}\mathrm{div}\mathcal{G}_{\mathbb{R}^3}(\lambda)\mathbf{f}_0\otimes\nabla\varphi \\
&\quad + \delta_{m,0}\nabla\left\{\varphi(-\Delta_{\Omega_b})^{-1}\mathrm{div}\mathcal{G}_{\Omega_{b}}(\mathbf{f}|_{\Omega_{b}})\right\},
\end{align*}
where $\delta_{m,0}=1$ if $m=0$ and $\delta_{m,0}=0$ if $m\ge1$.
Using Proposition \ref{R-R3-prop2} and \ref{R-R3-prop3} to the resultant identities, we have
\begin{align*}
&\left\|\frac{d^m}{d\lambda^m}\mathcal{R}^3(\lambda)\mathbf{f}\right\|_{Y^p(\Omega_{b_1} )} 
\le C\|\mathbf{f}\|_{X^p(\Omega)}
\begin{cases}
1 & (m=0), \\
0 & (m\ge1),
\end{cases} \\
&\left\|\frac{d^m}{d\lambda^m}\mathcal{R}^3(\lambda)\mathbf{f}\right\|_{Y^p(D_{b_1,b_2})} 
\le C\|\mathbf{f}\|_{X^p(\Omega)}
\begin{cases}
1 & (m=0), \\
\left|\log |\lambda|\right| & (m=1), \\
|\lambda|^{1-m} &(m\ge2),
\end{cases} \\
&\left\|\frac{d^m}{d\lambda^m}\mathcal{R}^3(\lambda)\mathbf{f}\right\|_{Y^p( D_{b_2,b})} 
\le C\|\mathbf{f}\|_{X^p(\Omega)}
\begin{cases}
1 & (m\le 1), \\
\left|\log |\lambda|\right| & (m=2), \\
|\lambda|^{2-m} &(m\ge3)
\end{cases}
\end{align*}
for $\lambda\in \Dot{U}_{\lambda_2}$ and $\mathbf{f}\in X_b^p(\Omega)$.
Therefore, combining the above estimates, Leibniz' rule 
\[
\frac{d^m}{d\lambda^m}\widetilde{\mathcal{R}}(\lambda)\mathbf{f}=\sum_{l=0}^m \binom{m}{l} \frac{d^l}{d\lambda^l}\mathcal{R}^3(\lambda)\frac{d^{m-l}}{d\lambda^{m-l}}(I+\mathcal{S}^3(\lambda))^{-1}\mathbf{f}
\]
and $\|(d^m/d\lambda^m)(I+\mathcal{S}^3(\lambda))^{-1}\mathbf{f}\|_{X^p(\Omega)}\le C\|\mathbf{f}\|_{X^p(\Omega)}$ for $\lambda\in \Dot{U}_{\lambda_2}$ and $\mathbf{f}\in X_b^p(\Omega)$, we arrive at \eqref{R-Ex-Prop2-3}--\eqref{R-Ex-Prop2-5}.
This completes the proof.
\end{proof}

\vspace{-7mm}

\section{Proof of Theorem \ref{Local-energy}} 

\vspace{-2mm}
In this section, we are going to derive \eqref{Local-energy-1} in Theorem \ref{Local-energy}. Throughout this section, we only show the case $m=0$ and impose the assumptions in Theorem \ref{Local-energy}. Observing from Proposition \ref{R-Ex-Prop1}, $A$ generates an analytic semigroup $T(t)$ expressed by
\begin{gather}\label{pr-Thm1-1}
T(t)u_0=\frac{1}{2\pi i}\int_{\Gamma_\varepsilon} e^{\lambda t}(\lambda+A)^{-1}u_0 d\lambda
\end{gather}
in $Y^p(\Omega)$ for $u_0\in X^p(\Omega)$ and $\varepsilon>0$, where $\Gamma_\varepsilon:=\{\lambda=\varepsilon+i s~;~s:-\infty\to\infty\}$. Furthermore, restricting $u_0\in Y^p(\Omega)$, \eqref{pr-Thm1-1} is reformulated as 
\begin{gather}\label{pr-Thm1-2}
T(t)u_0=-\frac{1}{2\pi i t}\int_{\Gamma_\varepsilon} e^{\lambda t}(\lambda+A)^{-2}u_0 d\lambda
\end{gather}
by noticing \eqref{R-Ex-prop1-2} in Proposition \ref{R-Ex-Prop1} with $\lambda_0=\varepsilon$.
We next approximate the integrand in the right-hand side of \eqref{pr-Thm1-2}. We take $\lambda_3$ from Corollary \ref{R-Ex-Cor1}. For $R,\delta,\varepsilon>0$ with $\delta<\lambda_3<R$, we set
\begin{align*}
\Gamma_{R,\varepsilon}^1&:=\left\{\lambda=\varepsilon+i s~;~s:-\left(R_1+R\cdot\mathrm{Im}(e^{\frac{3}{4}i\pi})\right)\to R_1+R\cdot\mathrm{Im}(e^{\frac{3}{4}i\pi}) \right\}, \\
\Gamma_{R,\varepsilon}^2&:=\left\{ \lambda=s+i\left(R_1+R\cdot\mathrm{Im}(e^{\frac{3}{4}i\pi})\right);~s:\varepsilon\to R\cdot\mathrm{Re}\left(e^{\frac{3}{4}i\pi}\right) \right\}, \\
\widetilde{\Gamma}_{R,\delta,\varepsilon}&:=\left\{\lambda=iR_1+se^{\frac{3}{4}i\pi };~s: R\to 0\right\} \cup \left\{\lambda=is;~s:R_1 \to \delta\right\}  \cup \left\{\lambda=\delta e^{i\theta};~\theta: \frac{\pi}{2}\to -\frac{\pi}{2} \right\} \\
&\qquad\quad \cup \left\{\lambda=is;~s: -\delta \to -R_1 \right\}  \cup \left\{\lambda=-iR_1-se^{\frac{3}{4}i\pi };~s: 0\to R\right\}, \\
\Gamma_{R,\varepsilon}^3&:=\left\{ \lambda=s-i\left(R_1+R\cdot\mathrm{Im}(e^{\frac{3}{4}i\pi})\right);~s: R\cdot\mathrm{Re}\left(e^{\frac{3}{4}i\pi}\right)\to\varepsilon \right\}.
\end{align*}
Here $R_1:=R_0+(\pi/2)$. 
Since $e^{\lambda t} (\lambda+A)^{-2}\in \mathcal{A}(\Xi_{\pi/4}; \mathcal{L}(Y^p(\Omega) ) )$ is seen from Proposition \ref{R-Ex-Prop1}, we have
\begin{gather}\label{pr-Thm1-3}
\begin{aligned}
\int_{\Gamma_{R,\varepsilon}^1}e^{\lambda t} (\lambda+A)^{-2} u_0 d\lambda &=-\int_{\widetilde{\Gamma}_{R,\delta,\varepsilon}}e^{\lambda t} (\lambda+A)^{-2} u_0 d\lambda\\
&\qquad-\int_{\Gamma_{R,\varepsilon}^2}e^{\lambda t} (\lambda+A)^{-2} u_0 d\lambda-\int_{\Gamma_{R,\varepsilon}^3}e^{\lambda t} (\lambda+A)^{-2} u_0 d\lambda.
\end{aligned}
\end{gather}
It is clear that the left-hand side of \eqref{pr-Thm1-3} tends to the integrand in the right-hand side of \eqref{pr-Thm1-2} as $R\to\infty$. Since we have $\Gamma_{R,\varepsilon}^2\cup \Gamma_{R,\varepsilon}^3 \subset \Xi_{\pi/4}\cap \{\lambda\in\mathbb{C}~;~|\lambda|\ge R_1\}$ for $\varepsilon>0$, the second and third terms in the right-hand side of \eqref{pr-Thm1-3} tend to 0 in $Y^p(\Omega)$ as $R\to\infty$ by using Proposition \ref{R-Ex-Prop1}.  Hence, it remains to investigate the first term in the right-hand side of \eqref{pr-Thm1-3} as $R\to\infty$ and $\delta\to 0$. We rewrite this as
\begin{align*}
&-\int_{\widetilde{\Gamma}_{R,\delta,\varepsilon}}e^{\lambda t} (\lambda+A)^{-2} u_0 d\lambda=\mathcal{I}_{\delta}^0(t)+\mathcal{I}_{ \delta}^{1 }(t)+\mathcal{I}_{R}^{2}(t), 
\end{align*}
where
\begin{align}
&\mathcal{I}_{\delta}^0(t):=i\delta \int_{|\theta|<\frac{\pi}{2}}e^{i\theta} e^{\delta e^{i\theta}t} (\delta e^{i\theta}+A)^{-2}u_0 d\theta, \label{pr-Thm1-4}  \\
&\mathcal{I}_\delta^{1 }(t):=-i\int_{\delta<|s|<R_1}e^{its} (is+A)^{-2}u_0 ds, \label{pr-Thm1-5} \\
&
\begin{aligned}
\mathcal{I}_R^{2}(t):=&\int_{0}^R e^{(iR_1+s e^{\frac{3}{4}\pi i}) t} (iR_1+s e^{\frac{3}{4}\pi i}+A)^{-2} u_0 ds -\int_{0}^R e^{-(iR_1+s e^{\frac{3}{4}\pi i}) t} (-iR_1-s e^{\frac{3}{4}\pi i}+A)^{-2} u_0 ds.
\end{aligned} \label{pr-Thm1-6}
\end{align}
It follows from Proposition \ref{R-Ex-Prop2} that \eqref{pr-Thm1-4} tends to 0 in $Y^p(\Omega_b)$ as $\delta\to0$. We next consider \eqref{pr-Thm1-5}.  Applying integration by parts and letting $\delta\to0$ in \eqref{pr-Thm1-5} yields
\begin{align}\label{pr-Thm1-7}
\mathcal{I}^{1}(t)&:=\lim_{\delta\to0}\mathcal{I}_\delta^{1 }(t)=\int_{-R_1}^{R_1}e^{is t} \mathfrak{J}(s)u_0 ds,
\end{align}
where $\mathfrak{J}(s)u_0:=(d/ds)(is+A)^{-1}u_0$. We make use of the following lemma to obtain the decay estimate of $\mathcal{I}^{1 }(t)$ in $Y^p(\Omega_b)$. 
\begin{lem}\label{pr-Thm1-lem}
Let $X$ be a Banach spaces with $|\cdot|_X$ and let $f=f(s)$ be a $X$-valued  function satisfying $f\in C(\mathbb{R}; X)\cap C^\infty(\mathbb{R}\setminus\{0\}; X)$ and $f(s)=0$ for $s\in\mathbb{R}$ with $|s|\ge \lambda_3$. Assume that there exists a constant $C(f)$ depending on $f$ such that
\[
\left|\frac{d^m}{d s^m}f(s)\right|_X \le C(f)
\begin{cases}
1 & (m=0), \\
\max\left\{1,\left|\log|s| \right| \right\} & (m=1),
\end{cases}
\]
and define $g(t):=\int_{-\infty}^\infty e^{is t}f(s)ds$. Then there exists a constant $C(\lambda_3)$ depending only on $\lambda_3$ such that $g$ is estimated as 
\[
|g(t)|_X\le C(\lambda_3) C(f) (1+|t|)^{-1}
\]
for $t\in \mathbb{R}$.
\end{lem}
\begin{proof}
We easily see that $|g(t)|_X\le C(\lambda_3)C(f)$ holds for $|t|\le 1$. Rewriting $g$ as $g(t)=\int_{|s|\le t^{-1}}e^{is t}f(s)ds+\int_{|s|\ge t^{-1}} e^{is t}f(s)ds$ and applying integration by parts to the last term once in the right-hand side of the resultant identity, we earn $|g(t)|_X\le C(\lambda_3)C(f) |t|^{-1}$ for $|t|\ge 1$. 
This completes the proof.
\end{proof}
Let $\chi\in C_0^\infty(\mathbb{R})$ be a cut-off function such that $\chi(s)=1$ for $|s|\le \lambda_3/2$ and $\chi(s)=0$ for $|s|\ge \lambda_3$. We then decompose \eqref{pr-Thm1-7} as $\mathcal{I}^{1 }(t):=\mathcal{I}^{1,1 }(t)+\mathcal{I}^{1,2 }(t)$, where
\begin{gather*}
\mathcal{I}^{1,1 }(t):=\int_{-\infty}^{\infty}e^{its} \chi(s)  \mathfrak{J}(s)u_0 ds,~
\mathcal{I}^{1,2 }(t):=\int_{\frac{\lambda_2}{2}\le|s|\le R_1}e^{its} (1-\chi(s)) \mathfrak{J}(s)u_0 ds.
\end{gather*}
It is seen from Corollary \ref{R-Ex-Cor1} that $\chi(s)  \mathfrak{J}(s)u_0$ satisfies the same assumptions as $f$ in Lemma \ref{pr-Thm1-lem} with $X=Y^p(\Omega_b)$ and $C(f)=C\|u_0\|_{X^p(\Omega)}$. Therefore employing this lemma, the estimate of $\mathcal{I}^{1,1 }(t)$ is given by
\begin{align}\label{pr-Thm1-8}
\|\mathcal{I}^{1,1 }(t)\|_{Y^p(\Omega_b)}\le Ct^{-1}\|u_0\|_{X^p(\Omega)}
\end{align}
for $t\ge1$. The estimate of $\mathcal{I}^{1,2 }(t)$ is derived as
\begin{align}\label{pr-Thm1-9}
\|\mathcal{I}^{1,2 }(t)\|_{Y^p(\Omega)}\le Ct^{-1}\|u_0\|_{X^p(\Omega)}
\end{align}
for $t\ge 1$ by applying integration by parts in its definition. Here we use the fact $1-\chi(s)=0~(|s|\le \lambda_3/2)$ and  Proposition \ref{R-Ex-Prop1} with $\lambda_0=\lambda_3/2$. Hence, \eqref{pr-Thm1-7}, \eqref{pr-Thm1-8} and \eqref{pr-Thm1-9} yield
\begin{align}\label{pr-Thm1-10}
\|\mathcal{I}^{1}(t)\|_{Y^p(\Omega_b)}\le Ct^{-1}\|u_0\|_{X^p(\Omega)}
\end{align}
for $t\ge 1$. We finally consider \eqref{pr-Thm1-6} as $R\to\infty$. Noting 
\begin{gather*}
e^{\pm(iR_1+s e^{\frac{3}{4}\pi i}) t}=\pm(te^{\frac{3}{4}\pi i})^{-1}\frac{\partial}{\partial s}e^{\pm(iR_1+s e^{\frac{3}{4}\pi i}) t},\quad 
\pm(iR_1+s e^{\frac{3}{4}\pi i})\in \Xi_{\frac{\pi}{4}}\cap \{\lambda\in\mathbb{C}~;~|\lambda|\ge R_1\}
\end{gather*}
for $s\ge0$, applying integration by parts and using Proposition \ref{R-Ex-Prop1} with $\lambda=\pm(iR_1+s e^{\frac{3}{4}\pi i})$, $\varepsilon_0=\pi/4$ and $\lambda_0=R_1$, the limit $\mathcal{I}^{2}(t):=\lim_{R\to\infty}\mathcal{I}_R^{2}(t)$ in $Y^p(\Omega)$ exists and we have
\begin{align}\label{pr-Thm1-11}
\|\mathcal{I}^{2}(t)\|_{Y^p(\Omega)}\le Ct^{-1}\|u_0\|_{X^p(\Omega)}
\end{align}
for $t\ge 1$. As a result, letting $\delta\to0$ and $R\to\infty$ in \eqref{pr-Thm1-3} and employing \eqref{pr-Thm1-10} and \eqref{pr-Thm1-11}, \eqref{Local-energy-1} with $m=0$ is proved for $u_0\in X_b^p(\Omega)\cap Y^p(\Omega)$. Due to the density of $Y^p(\Omega)$ in $X^p(\Omega)$, we are able to relax this condition for $u_0$ as $u_0\in X_b^p(\Omega)$. This completes the proof of Theorem \ref{Local-energy}. $\blacksquare$

\vskip1mm

\noindent
\textbf{Acknowledgements.}  
Y. I. was partially supported by JSPS KAKENHI Grant Number JP23KJ1408. T. K. was partially supported by JSPS KAKENHI Grant Number JP22K03374. \\[-6ex]


\end{document}